\documentclass{article}

\usepackage{geometry}
\usepackage[mathcal]{euscript}
\usepackage{hyperref}
\usepackage{amsmath,amssymb,amsthm}
\usepackage{mathtools}
\usepackage{dutchcal}

\DeclareMathAlphabet{\mathpzc}{OT1}{pzc}{m}{it}

\usepackage{xcolor}

\newtheorem{thm}{Theorem}[section]
\newtheorem{lem}[thm]{Lemma}
\newtheorem{prop}[thm]{Proposition}
\newtheorem{cor}[thm]{Corollary}
\theoremstyle{remark}
\newtheorem{rem}[thm]{Remark}

\theoremstyle{definition}

\newtheoremstyle{Claim}{}{}{\itshape}{}{\itshape\bfseries}{:}{ }{#1}
\theoremstyle{Claim}

\newcommand{\T}{{\mathbb{T}^d}}

\newcommand{\U}{\mathcal{U}}
\newcommand{\M}{\mathcal{M}}

\newcommand{\R}{\mathbb{R}}
\newcommand{\Z}{{\mathbb{Z}}}

\newcommand\Mcal{{\mathcal{M}}}
\newcommand\Ucal{{\mathcal{U}}}

\def\be{\begin{equation}}
\def\ee{\end{equation}}
\def\dive{{\rm div}}
\def\rife#1{\eqref{#1}}
\def\rg{\rangle} 
\def\lg{\langle} 

\def\cP{{\mathcal P}}
\def\PT{\cP(\T)}

\def\vfi{\varphi}

\def\into{\int_{\T}}

\def\dive{{\rm div}}
\def\de{\delta}

\def\vep{\varepsilon}
\def\elle#1{L^{#1}(\T)}

\newcommand{\eps}{\varepsilon}

\numberwithin{equation}{section}

\titlepage
\title{Long time behaviour and turnpike solutions in mildly non-monotone mean field games}
\author{Marco Cirant\footnote{Dipartimento di Matematica, Universit\`a di Padova, Via Trieste 63, 35121–Padova, Italy, email: cirant@math.unipd.it}\quad 
 and  \quad Alessio Porretta\footnote{Dipartimento di Matematica, Universit\`a di Roma Tor Vergata, Via della Ricerca Scientifica 1, 00133 Roma, email: porretta@mat.uniroma2.it.  }}

\begin{document}

\maketitle

\begin{abstract} 
We consider mean field game systems in time-horizon $(0,T)$, where the individual cost functional depends locally on the density distribution of the agents, and the Hamiltonian is locally uniformly convex. We show that, even if the coupling cost functions are mildly non-monotone, then the system is still well posed due to the effect of individual noise. The rate of anti-monotonicity (i.e. the aggregation rate of the cost functions) which can be afforded depends on the intensity of the diffusion and on global bounds of solutions. We give applications to either the case of globally Lipschitz Hamiltonians or the case of quadratic Hamiltonians and couplings having mild growth. 

Under similar conditions, we investigate the long time behavior of solutions and we give a complete description of the ergodic and long term properties of the system. In particular we prove: (i) the turnpike property of solutions in the finite (long) horizon $(0,T)$, (ii) the convergence of the system from $(0,T)$ towards  $(0,\infty)$, (iii) the vanishing discount limit of the infinite horizon problem and the long time convergence towards the ergodic stationary solution. 

This way we extend previous results which were known only for the case of monotone and smoothing couplings; our approach is self-contained and  does not need the use of the  linearized system or of the master equation.
\end{abstract}

%
\section{Introduction}

The theory of mean field games was initiated by J.-M. Lasry and P.-L. Lions since 2006  (\cite{LL06cr1}, \cite{LL06cr2}, \cite{LL07}) in order to describe Nash equilibrium configurations in multi-agents strategic interactions. Similar ideas appeared in \cite{HCM}, at the same time. Mostly peculiar to the approach suggested by Lasry and Lions is to describe the equilibria through solutions of a forward-backward system of PDEs involving a Hamilton-Jacobi-Bellman equation for the value function  of a  single agent and a  Kolmogorov-Fokker-Planck equation for the distribution law  of the whole population. The simplest form of this kind of system is the following 
\begin{equation}\label{mfgT}
\begin{cases}
-u_t - \kappa\Delta u + H(x, Du) = F(x, m(t)), & t 	\in (0, T) \\
m_t - \kappa \Delta m - {\rm div}(m H_p(x, Du)) = 0, & t \in (0, T) \\
m(x, 0)=m_0(x), \qquad u(x, T) = G(x, m(T))
\end{cases}
\end{equation}
where $u(t,x)$ represents the optimal value for a player at time $t$ in state $x$, while $m(t,x)$ represents the density of the distribution law of the controlled dynamical state.
The typical interpretation of the system is that any single player controls his/her dynamical state (in a standard probability space where  a given $d$-dimensional Brownian motion  is defined) 
$$
\begin{cases}
  dX_\tau= \alpha_\tau\, d\tau + \kappa\, dB_\tau &  \\
X_t= x &   
\end{cases}
\rightsquigarrow \quad u(t,x)= \inf_{\{\alpha_\tau\}} {\mathbb E} \left\{ {\small \int_t^T [L(X_\tau, \alpha_\tau)+ F(X_\tau,m_\tau)] + G(X_T, m_T)} \right\}
$$
where $\kappa>0$ and $L,F,G$ represent different costs depending on the control process $\{\alpha_\tau\}$ as well as on the measures $\{m_\tau\}$. Here $m$ enters as an exogeneous
datum and stays fixed in the agents'optimization; so defining  the Hamiltonian $H(x,p):= \sup_{q} [-q\cdot p- L(x,q)]$, under suitable convexity and smoothness conditions the value function satisfy the first equation in \rife{mfgT} and $\alpha_t= -H_p(X_t,Du(t,X_t)) $ is the optimal feedback control. Assuming consistency with  the agents' rational anticipations, at equilibrium it happens that the distribution law of the optimal controlled process coincides with the family of measures  $m_t$   used in the individual optimization. Thus $m$ satisfies the second equation as the law of the optimal process, where $\alpha$  corresponds to the optimal feedback  strategy.
We refer to \cite{CP-CIME} for an extended introduction to mean field games systems. 

In all the above presentation, the state space could be differently chosen, together with possibly boundary effects (reflection or absorption effects at the boundary, for instance)
but we will assume here the simplest, yet instructive case of periodic setting. This means that $x$ belongs to  the flat torus $\T= \R^d/{\Z^d}$.

This is a typical setting to investigate ergodic properties and long time behavior of the controlled dynamics. 
This kind of question is very natural in control theory and quite popular in applications. Many economic models, for instance, expect that if the time horizon is long then  the optimal strategies are nearly stationary for most of the time.  This is referred to as the {\it turnpike property} of optimal control problems, according to a terminology 
introduced by P.A. Samuelson in 1949 (see \cite{DSS}). Even if   the {\it turnpike theory} is a longstanding topic in optimal control, it has attracted  an increasing  renovated interest in the last years from both theoretical and applicative viewpoint: it is impossible here to mention all contributions in this direction, we refer  e.g. to \cite{Grune+al}, \cite{PZ}, \cite{TZ}, \cite{Zu} and references therein.  
\vskip0.3em
When coming at mean field game  systems as in \rife{mfgT}, the long time behavior was investigated  in several papers under the assumption that the cost functions $F,G$ are nondecreasing in $m$. It is well established in the theory that this monotonicity condition gives uniqueness and stability of solutions; under the same condition the turnpike property   and the convergence of solutions towards the stationary ergodic state have been first proved 
in  \cite{CLLP1}, \cite{CLLP2} for quadratic Hamiltonian (i.e. $H(x,p)= |p|^2$). Milder statements  (time average convergence) or stronger statements (exponential pointwise decay estimates) were obtained according to  different  sets of assumptions.
The case of discrete time, finite states system was analyzed in \cite{gomes2010discrete}. Later on, the long time behavior was completely described  in \cite{CaPo} in case  of smoothing couplings and uniformly convex Hamiltonian, and in \cite{P-MTA} for the case of local couplings and globally Lipschitz Hamiltonian.   We also point out that the turnpike pattern is clearly shown in many numerical simulations, see e.g. \cite{Achdou-CIME}.
\vskip0.3em
The purpose of this paper is twofold. On one hand we wish to show that the monotonicity of the couplings $F,G$ can be relaxed to some extent, using the diffusive character of the equations. Otherwise said, we show that {\it the Brownian noise in the individual dynamics can compensate, to some extent, the lack of monotonicity} so that mildly aggregative cost functions $F,G$ may not affect the  uniqueness of solutions and their stability, even in long time. 
This was already suggested by P.-L. Lions in the early stages of the theory (\cite{L-college}) although we could not find any further development of this issue in the literature.
On another hand,  still in the  context of monotone or mildly non-monotone couplings, we revisit both the long time behavior of solutions and the ergodic limits in order to clarify the full picture: turnpike estimates for solutions in $(0,T)$, pointwise limit of the system  as $T\to \infty$, vanishing discount limit in the infinite horizon problem.  

The above   issues are somehow related. In fact, the complete characterization of the long time behavior, namely the convergence of $u^T(t)-\lambda (T-t)$ and $m^T(t)$ at any time $t$, was previously obtained in \cite{CaPo} as a byproduct of the long time convergence of the master equation. The master equation is an infinite dimensional equation 
(defined on time, space and the Wasserstein space of probability measures) which gives the value function $u$ as a feedback of the measure $m$. This equation is not easy
to handle and plays a similar role as the Riccati equation for the feedback operator in control systems. Establishing the long time behavior of the master equation allows one to 
completely describe the behavior of the system but   is actually a hard result which requires much stronger conditions, starting from the  smoothness of the functions $F,G$.
In addition, what is more relevant here,  when the couplings are not monotone, the master equation can not be properly used, since no satisfactory notion of solution has been developed so far outside the   monotone case.  Therefore, motivated by a setting of mildly nonmonotone cost functions $F,G$,  we refine and develop some of the arguments  introduced in  \cite{CaPo} in order to study the long time convergence  - as well as the vanishing discount limit - {\it without any use of the master equation}.  The approach that we propose here is actually simpler and only  relying on the PDE system; this seems to be much more flexible and it is actually promising for many other situations where the master equation still looks untractable (e.g. the case of state constraint  problems, congestion models, etc...).

Let us mention that other nonmonotone mean field games have been considered in several works previously: for example,  second-order problems have been analyzed in \cite{cjde, CGhi, Ci-To}, while \cite{CeCi, GoSe} deal with first-order systems. The works by Ambrose (see \cite{Ambrose} and references therein) show the existence of solutions to second-order systems under smallness conditions on the coupling, though these conditions seem to be depending on the time horizon $T$. Tran considers in \cite{Tran} a setting which is nonmonotone in a broader sense: the coupling in the system is increasing, but the Hamiltonian is not convex. Though he does not address the long-time analysis, his uniqueness results are closer to ours, as he obtains uniqueness under smallness conditions on $H$, which are independent of the time horizon.
\vskip0.3em
Let us now summarize a bit more precisely our results, with reference to the content of the next Sections.
 
 \begin{itemize}
 
 \item  In Section 3, we first show that, for given $L^\infty$- bounds in $(0,T)$ on $m$ and $Du$, there is some $\gamma$ (depending on the diffusion constant $\kappa$, but not on $T$) such that if $F+ \gamma m$ and $G+ \gamma m$ are nondecreasing  then the system \rife{mfgT} has a unique solution. See Theorem \ref{uniq}.
 
 \end{itemize}
 
\noindent A similar result   is also proved for the stationary  ergodic problem  
\begin{equation}\label{MFGergo}
\begin{cases}
\bar \lambda - \kappa \Delta \bar u + H(x, D\bar u) = F(x, \bar m), \\
- \kappa \Delta \bar m - {\rm div}(\bar m H_p(x, D\bar u)) = 0, \\
\int_\T \bar m = 1, \quad \int_\T \bar u = 0.
\end{cases}
\end{equation}
We suggest  two sets of assumptions under which  this kind of result can be applied,   assuming for simplicity that the final cost   is a given function $u_T(x)$:

\begin{itemize}

\item[(A)]  the case of  globally Lipschitz (and locally uniformly convex) Hamiltonian. 

This corresponds to typical control problems with smooth, uniformly convex, Lagrangian cost  and controls in a compact set.  Thanks to the global Lipschitz bound of $H$, in this case the growth of $F$   can be arbitrary,  for instance one can take $F=-\gamma\, m^\alpha$.  Then  the system \rife{mfgT} is well-posed if $\gamma$ is not too large  (see also Remark  \ref{quanto}). 

\item[(B)] the case of superlinear, uniformly convex, Hamiltonian, with quadratic-like growth, and cost function  $F $ which satisfies $0\geq F(x,m)\geq - \gamma\, m^\alpha$ with $\alpha<\frac 2{d}$.

The  threshold  $\frac 2d$ for the growth of the coupling is not new in the context of mean field game systems with quadratic Hamiltonian (see e.g. \cite{CGhi}, \cite{Ci-To}, \cite{Gomes-book}), and we comment this issue in Remark \ref{remgrow}. 

\end{itemize}

Let us  stress that, in the aforementioned results,  the admissible threshold $\gamma$ of anti-monotonicity  depends on the diffusivity constant $\kappa$ and on the initial datum (through $\|m_0\|_\infty$). This latter fact explains very well why  the master equation is hardly usable,  in this context; indeed, there is no general feedback policy of $u$ in the space of probability measures (unless we reduce to the case of monotone couplings).

 \begin{itemize}
 
\item  In Section 4 we show that, in the same context given above, the solutions $(u^T,m^T)$ of \rife{mfgT} satisfy an exponential turnpike property. More precisely,
if $\|Du^T\|_\infty$ is   bounded in $(1,T-1)$ independently of $T$, then there  exists $\omega>0$ and $M$ (independent of $T$) such that 
$$
\| m^T(t)-\bar m\|_\infty + \| Du^T(t)-D\bar u\|_\infty \leq M (e^{-\omega t}+ e^{-\omega (T-t)}) \qquad \forall t\in (1,T-1)\,,
$$
 where $(\bar u, \bar m)$ is a stationary state. See Theorem \ref{longtime} and Corollary \ref{boundT}.
 
 \end{itemize}
 
 In particular, this result applies to the examples (A), (B) mentioned above. It is to be noted that this result is independent of initial and terminal conditions, as is customary in turnpike theory.
 
 \begin{itemize}
 \item In Section 5, we describe the convergence of $u^T(t), m^T(t)$ {\it at any time scale $t$}. This is now influenced by both initial and terminal conditions   (which we assume to be fixed).  Namely, we prove that, given $m_0\in \elle\infty,  u(T)\in W^{1,\infty}(\T)$ and, for instance, a globally Lipschitz, locally uniformly convex, Hamiltonian, there exists $\gamma >0$  such that,  if $F + \gamma  m$ is nondecreasing, then 
 $$
 u^T(t,x)-\bar \lambda (T-t) \mathop{\to}^{T\to \infty} v(t,x)\,,\qquad m^T(t,x) \mathop{\to}^{T\to \infty}  \mu(t,x)
 $$
 locally uniformly in $[0,\infty)\times \T$, where   $(v,\mu)$ is one particular solution of the infinite horizon problem
\be\label{particu}
 \begin{cases}
-v_t + \bar \lambda - \kappa\Delta v + H(x, D v)= F(x,\mu )\,, & \hbox{$t\in (0,\infty)$}
\\
\mu_t- \kappa \Delta \mu -\dive( \mu\, H_p(x,D v))= 0\,,& \hbox{$t\in (0,\infty)$}
\\
\mu(0)= m_0 \,, \quad  v\in L^\infty((0,\infty)\times \T)\,,\,\, Dv \in D\bar u + L^2((0,\infty);\elle2)\,.\quad 
&  
\end{cases}
\ee
 Notice that the convergence is not just for subsequences, but for the whole sequence $(u^T, m^T)$. See Theorem \ref{convuT}.

 \item In Section 6 and 7, we describe the vanishing discount limit for the (discounted) infinite horizon problem
 $$
 \begin{cases}
-u_t + \de u -   \kappa \Delta u + H(x,Du)= F(x,m) & 
\\
m_t -  \kappa  \Delta m- \dive(m\, H_p(x,Du)) = 0 & 
\\
m(0)= m_0\,,\qquad u\in L^\infty((0,\infty)\times \T)\,. 
& \end{cases}
$$
Under similar conditions as before, we prove that the solution $(u_\de, m_\de)$ satisfies
$$
u_\de(t,x)- \frac{\bar \lambda}\de \mathop{\to}^{\de \to 0}   v(t,x) \,\,\,; \qquad \quad m_\de(t,x )\mathop{\to}^{\de \to 0} \mu(t,x) 
$$
 locally uniformly in $[0,\infty)\times \T$, where $(v,\mu)$ is  again one particular solution of  \rife{particu}. 

Once more, the convergence occurs for the whole sequence; we prove that actually the particular solution   selected in the limit satisfies
$$
v(t,x) \mathop{\to}^{t\to \infty} \bar u(x) + \theta\,\,\,, \,\,\,  
$$
where $\theta$ is itself a specific ergodic constant of a linearized problem. See Proposition \ref{ergdisc} and Theorem \ref{vanlim}.
 \end{itemize}
 
In the end, the results in Sections 6-7 establish the commutation property between the limit as $t\to \infty$ and the limit as $\de\to 0$ in the discounted infinite horizon problem. 
Let us  point out that $\bar m$ is the unique invariant measure of problem \rife{particu}; indeed, in this problem   $\mu$ is uniquely determined (while $v$ is unique up to addition of constants) and satisfies
$$
\mu(t,x) \mathop{\to}^{t\to \infty} \bar m(x)\qquad \hbox{uniformly in $\T$.}
$$
To conclude, we recall  that the results described in the above items and contained in Sections 4-7   were previously proved  in \cite{CaPo} for the case of smoothing and monotone couplings $F,G$,  using  the long time convergence of the master equation. Even if we rely on many ideas contained in \cite{CaPo}, we develop here a simpler
program to achieve those results, which avoids both the use of the master equation and the explicit use of the linearized mean field game system. The main benefit of this approach is that it is much less demanding on the functions $F,G$ (say, much cheaper in terms of the required smoothness) and more case-sensitive in terms of initial conditions. We give evidence of this fact by extending the results of \cite{CaPo} to the case of  couplings which are local and  even possibly non-monotone.

Let us point out that   even the case of {\it non local } mildly non-monotone couplings   could be dealt with in a similar way but we did not pursue this extension here for the sake of simplicity.

\section{Standing assumptions}
Let $x$ belong to the flat torus $\T$. We denote by $\PT$ the space of probability measures on $\T$. 
Throughout the whole paper, we suppose that 
$H(x, \cdot)$ is  locally Lipschitz continuous and locally uniformly convex on $\R^d$. Namely, $p\mapsto H(x,p)$ is a $C^2$  function which satisfies
\be\label{liploc}
\forall K>0\,,\,\, \exists \,L_K>0\,:\quad  |H_p(x,p)|\leq L_K \qquad \forall (x,p)\in \T \times \R^d\,: \, |p|\leq K
\ee  
and
\be\label{Hk}
\forall K>0\,,\,\, \exists \,\alpha_K, \beta_K>0\,:\quad \alpha_K\, I  \leq H_{pp}(x,p) \leq \beta_K I \qquad \forall (x,p)\in \T \times \R^d\,: \, |p|\leq K\,.
\ee
The couplings $F,G$ are real valued functions defined on $\T\times [0,\infty)$;  we   suppose, as a standing condition, that they are locally bounded and Lipschitz continuous with respect to $m$:
\be\label{fK}
 \forall \,K>0\,,\,\, \exists \,\, c_K\,,\,\ell_K>0\,: \,  \begin{cases}  
 |F(x,m)|\leq  c_K \,, &   \qquad  \forall x \in \T \,, m,m' \in \R:
\\
 |F(x,m)-F(x,m')|\leq  \ell_K |m-m'|  &  \qquad  |m|, |m'|\leq K\, 
 \end{cases}
\ee
and similarly for $G(x,\cdot)$:
\be\label{gK}
 \forall \,K>0\,,\,\, \exists \,\, \hat c_K\,,\,\hat \ell_K>0\,: \,  \begin{cases}  
 |G(x,m)|\leq  \hat c_K \,, &   \qquad  \forall x \in \T \,, m,m' \in \R:
\\
 |G(x,m)-G(x,m')|\leq  \hat \ell_K |m-m'|  &  \qquad  |m|, |m'|\leq K\, 
 \end{cases}
\ee

The dependence of $H,F,G$ with respect to $x$ is only assumed to  be measurable; as it is commonly said, they are Carath\'eodory functions (measurable in $x$ for any $p$, or $m$, accordingly), and the above conditions \rife{liploc}-\rife{gK} are meant to hold almost everywhere for $x\in \T$. We stress that $H,F$ could depend on $t$ as well (in a measurable way) without additional difficulty, unless for results in which the long time behavior is concerned,  where this dependence could change drastically the picture, of course.

 
\section{Uniqueness for the MFG system}

In  a first result, we show that the MFG system admits a unique solution if the  rate of anti-monotonicity does not exceed some threshold, depending on the global $L^\infty$-bounds of $m,Du$.

To this purpose, we consider the set of (classical) solutions $(u, m)$ to \eqref{mfgT} which satisfy
\begin{equation}\label{MUT}
\sup_{[0, T] \times \T} m\le \mathcal M, \qquad \sup_{[0, T] \times \T} |D u| \le \mathcal U
\end{equation}
for some $\mathcal M, \mathcal U > 0$.


\begin{thm}\label{uniq} Let us set 
$$
X:=\{ (u, m)\in L^\infty((0,T);W^{1,\infty}(\T)) \times L^\infty((0,T)\times \T)\,\, \hbox{satisfying \rife{MUT}}\}.
$$  
There exists  $\gamma > 0$, only depending on $\kappa, \mathcal M, \mathcal U$ (in particular through the constants $L_{\U}, \alpha_\U, \beta_\U$ in \rife{liploc}-\rife{Hk}), such that if 
\[
\hbox{$F(x,m)+ \gamma m$,  $G(x,m)+ \gamma m$ are nondecreasing, \ \ for $ m \in [0, \mathcal M], x\in \T$,}
\]
then \eqref{mfgT} admits at most one solution $(u,m)$ in $X$. 
\end{thm}

\begin{proof}
Let $(u_1, m_1)$ and $(u_2, m_2)$ be two  solutions to \eqref{mfgT} which belong to the set $X$. Convexity of $H$ and the usual duality identity give
\begin{multline*}
- \frac d{dt} \into (u_1 - u_2)(m_1 - m_2) \ge \into m_2 \{ H(x, Du_1) - H(x, Du_2) - H_p(x, Du_2)D(u_1 - u_2)  \} \\ + \into m_1 \{ H(x, Du_2) - H(x, Du_1) - H_p(x, Du_1)D(u_2 - u_1)  \}+ \into [F(x, m_1) - F(x, m_2)][m_1 - m_2].
\end{multline*}
Integrating on $(0, T)$, using \rife{Hk} and the monotonicity assumption on $F(m),G(m)$, we have 
\begin{equation}\label{eq01}
\alpha_\U \int_0^T \into (m_1+m_2) |Du_1 - Du_2|^2 \le \gamma \int_0^T \into (m_1 - m_2)^2 + \gamma \into (m_1(T)-m_2(T))^2.
\end{equation}
The equation for $\rho := m_1 - m_2$ reads
\[
\rho_t - \kappa \Delta \rho - {\rm div}(\rho H_p(x, Du_1)) = {\rm div}\Big(m_2 \big(H_p(x, Du_1) - H_p(x, Du_2) \big) \Big), \quad t \in (0, T) 
\]
with $\rho(0) = 0$. Thus Lemma \ref{stimaRhoF} applies and yields,  by Lipschitz regularity of $H_p$
\begin{align*}
\int_0^T \into (m_1 - m_2)^2 dt  & \le C\, \beta_\U^2  \int_0^T \into (m_2)^2 |Du_1 - Du_2|^2 dt
\\
& \leq 
C\, \beta_\U^2\, \M   \int_0^T \into m_2 |Du_1 - Du_2|^2 dt
\end{align*}
for some constant $C$ only depending on $\kappa$ and $L_\U$ given by \rife{liploc}.

Similarly, using \rife{76} in Lemma \ref{stimaRhoF} we have
$$
\into  (m_1(T) - m_2(T))^2 \leq C \, \beta_\U^2\, \M  \int_0^T \into m_2 |Du_1 - Du_2|^2 dt
$$
Plugging those informations into \eqref{eq01}  we obtain
\[
\int_0^T \into m_2 |Du_1 - Du_2|^2 \le 2 \gamma\, C\, \beta_\U^2\, \alpha_\U^{-1}\M    \int_0^T \into m_2 |Du_1 - Du_2|^2 dt,
\]
and therefore $Du_1 = Du_2$ whenever $\gamma <  \alpha_\U (2C\, \beta_\U^2 \M)^{-1}$ (note that $m_2$ is bounded away from zero on $(0,T)$ by the strong maximum principle). The equalities $u_1 = u_2$ and $m_1 = m_2$ then follow by uniqueness of solutions of the Fokker-Planck equation and of the (backward) Bellman equation.
\end{proof}

\begin{rem}\label{remkappa} The constant $\gamma$ does not depend directly on the time horizon $T$, but only on $\M,\U$. In particular, if those bounds are independent of $T$, then so is $\gamma$ as well.
Note also that $\gamma$ depends on the diffusion coefficient $\kappa$, and it must vanish as $\kappa$ vanishes. Indeed, using bifurcation arguments as in \cite{cjde}, it is possible to prove the existence of multiple solutions (having comparable Lipschitz bounds) for large $T$ and arbitrarily small anti-monotonicity degree, i.e. $F' \approx -\kappa$. Therefore, the constant $C$ in the previous proof must explode as $\kappa \to 0$. 
\end{rem}

\begin{rem} Of course the constant $\gamma$ depends on the $L^\infty$-norm of the initial datum, because  $\|m_0\|_\infty\leq \M$.
\end{rem}

\vskip1em

A similar result can be proved to hold for the stationary ergodic problem \rife{MFGergo}.

\begin{thm}\label{uniq-erg} Let us set 
$$
X:=\{ (u, m)\in  W^{1,\infty}(\T) \times L^\infty( \T)\,:\, \|Du\|_\infty\leq \bar \U\,,\, \|m\|_\infty \leq \bar \M\}.
$$  
There exists $\gamma_0 > 0$, only depending on $\kappa, \bar \M, \bar\U$ (in particular through the constants $L_{\bar \U}, \alpha_{\bar \U}, \beta_{\bar \U}$ in \rife{liploc}-\rife{Hk}), such that if 
\[
\hbox{$F(x,m)+ \gamma_0 \,m$ is nondecreasing, \ \ for $ m \in [0, \bar \M], x\in \T$,}
\]
then \eqref{MFGergo} admits at most one solution $(\bar \lambda, u,m)$ in $X$. 
\end{thm}

\proof The proof follows the lines of Theorem \ref{uniq}.  If $(\bar \lambda_1, u_1, m_1)$ and $(\bar \lambda_2, u_2, m_2)$ are two  solutions to \eqref{MFGergo} which belong to  $X$, we consider the duality between the equations of $u_1-u_2$ and $m_1-m_2$. Since $m_1-m_2$ has zero average, the term  with $\lambda_1-\lambda_2$ disappears 
if integrated against $m_1-m_2$. Then one gets
\begin{align*}
&  \into m_2 \{ H(x, Du_1) - H(x, Du_2) - H_p(x, Du_2)D(u_1 - u_2)  \} \\ 
& + \into m_1 \{ H(x, Du_2) - H(x, Du_1) - H_p(x, Du_1)D(u_2 - u_1)  \}
\\
& \qquad\qquad + \into [F(x, m_1) - F(x, m_2)][m_1 - m_2] \leq 0
\end{align*}
which yields
\be\label{torno}
 \alpha_{\bar \U} \into m_2 |Du_1-Du_2|^2 \leq \gamma_0  \into |m_1 - m_2|^2  
\ee
Using \cite[Corollary 1.3]{CaPo} for the equation of $m_1-m_2$, we have, for some $C$ only depending on $\kappa,L_{\bar \U}$,
$$
\|m_1-m_2 \|_{\elle2}^2 \leq C \| m_2(H_p(x, Du_1)-H_p(x, Du_2))\|_{\elle2}^2 \leq C\, \bar \M\, \beta_{\bar U}^2\|\sqrt{m_2}(Du_1-Du_2)\|_{\elle2}^2
$$ 
Plugging this information into \rife{torno} gives that $Du_1-Du_2=0$ if $\gamma_0$ is sufficiently small, only depending on $\kappa, L_{\bar \U}, \alpha_{\bar \U}, \beta_{\bar \U}$. Since $Du_1=Du_2$ it follows that  $m_1=m_2$ (from the second equation) and $u_1=u_2$ from the prescribed normalization condition. Finally,  the first equation gives $\lambda_1=\lambda_2$.  
\qed

\vskip1em
We now give two  examples  of applications of the previous results, namely two settings where the global bounds \rife{MUT} are proved to hold.

\paragraph{(A) Globally Lipschitz Hamiltonians}
\ 
\vskip0.3em
The simplest case where  a global bound in ensured for $m,Du$ is when the Hamiltonian is globally Lipschitz. This is for instance the case when the set of controls of the individual agents lies in a compact set.  For simplicity, we assume  here that the final datum $G$ is $m$-independent, i.e.
\be\label{gut}
G(x, m(T)) = u_T(x) \in W^{1,\infty}(\T).
\ee

\begin{cor}\label{corlip} Assume conditions \rife{Hk}-\rife{fK}, and in addition suppose that $H(x,p)$ satisfies
\be\label{lip}
\exists \,\, L>0\, \,: \, \quad |H_p(x,p)|\leq L \qquad \forall (x,p)\in \T \times \R^d\,.
\ee  
Assume also that $G$ satisfies \rife{gut}.

For any $m_0\in \elle\infty$, there exist  a constant $\M_0$, only depending on $\kappa, L, \|m_0\|_\infty$ and another constant $\gamma_0$, depending on $\kappa, L, \|m_0\|_\infty, \|u_T(x)\|_{W^{1,\infty}(\T)}$ and on the functions $F,H$ (through the constants in \rife{Hk}-\rife{fK} for  a value of $K$ depending on $\kappa, L, \|m_0\|_\infty,\|u_T(x)\|_{W^{1,\infty}(\T)}$) such that if  
\be\label{fgm0}
\hbox{$F(x,m)+ \gamma_0 m$  is nondecreasing, \ \ for $ m \in [0, \mathcal M_0], x\in \T$,}\ee
then the MFG system \rife{mfgT} has a unique solution.
\end{cor}

\begin{proof} Since $\|H_p(x, Du)\|_\infty \leq L$ due to \rife{lip}, by standard parabolic regularity we know that there exists $\M_0$, only depending on $\kappa,L, \|m_0\|_\infty$ such that
$$
\|m\|_\infty \leq \M_0\,.
$$ 
Using \rife{gut}
and parabolic regularity, there exists a constant $\U_0$, depending on $\kappa, L$,  $\|u_T\|_{W^{1,\infty}(\T)} $ and on the regularity of $F(x,m)$ for $|m|\leq \M_0$, such that
$$
\| Du\|_\infty \leq \U_0\,.
$$
Since the bounds above are true for all solutions, by Theorem \ref{uniq} there exists $\gamma_0$, depending on $\kappa, \M_0, \U_0$, such that if \rife{fgm0} holds then
there is a unique solution of the MFG system. This concludes the proof of the statement, because  $\M_0, \U_0$ only depend on $\kappa,L, \|m_0\|_\infty, \|u_T(x)\|_{W^{1,\infty}(\T)}$ and on the local behavior of $F, H$ on related compact sets.  
\end{proof}

\begin{rem}\label{quanto} It is possible to quantify a bit more precisely the dependence of $\gamma_0$ on the functions $F,H$. Assume that $H$ satisfies \rife{lip} and  
$$
c_0 \, (1+|p|)^{-1} I_d\leq H_{pp} \leq c_0^{-1} I_d\qquad \forall (x,p)\in \T\times \R^d\,
$$
for some $c_0>0$.  Let  for simplicity $G=0$,  and suppose that  $F(x,m)\simeq -\gamma\, m^p$, with $|F_m(x,m) |\simeq \gamma \, m^{p-1}$ for some $\gamma>0$. 

If $L$ is given by \rife{lip}, then one has 
$$
\|m\|_\infty \leq \mathcal M= C(\kappa, L) \|m_0\|_\infty
$$
for a constant $C$ only depending on $\kappa, L$ (and the dimension $d$). Therefore $\|F(x,m)\|_\infty \lesssim \gamma\, \M^p$, hence
$$
\|Du\|_\infty \leq \mathcal U = C(\kappa, L)\gamma\,  \M^p
$$
for a possibly different constant $C$ still depending only on $\kappa, L, d$.  Coming back to the proof of Theorem \ref{uniq}, with $\alpha_U\simeq \frac{c_0} {\U}, \beta_U= c_0^{-1}$ we need to require  
$$
\gamma  \lesssim C\, \frac1{\U\, \M} \lesssim C\, \frac1{\gamma\, \, \M^{p+1}}\,.
$$
Hence we can estimate 
$$
\gamma_0 \lesssim C(\kappa, L) \, \|m_0\|_\infty^{-(p+1)/2}
$$ 
for some $C$ only depending on $\kappa, L, d$. It is also easy to check that $C\to 0$ as $\kappa \to 0$, so we have 
$$
\gamma_0 \to 0 \qquad \hbox{as $\kappa\to 0$ or $\|m_0\|_\infty\to \infty$.}
$$
One may guess at this point that $\gamma_0$ vanishes as the diffusion $\kappa$ vanishes because norms of solutions are uncontrolled. In fact this seems related to subtler issues, involving the deterioration of the exponential decay of the heat semigroup as $\kappa \to 0$. Indeed, as we already observed in Remark \ref{remkappa}, one can find multiple solutions with controlled norms whenever $F' \approx -\kappa$. Hence, even though one assumes bounds $\U, \M$ on solutions, it is mandatory to require smaller $\gamma_0$ as $\kappa$ becomes smaller. Note also that bifurcation methods allow to construct solutions that are periodic in time; therefore, not only uniqueness fails, but also the turnpike property that will be addressed in the next section, at least for selected families of solutions. See also \cite{GoSe} for the failure of long time stabilization (due to existence of traveling waves) in deterministic mean field games with anti-monotone couplings.
\end{rem}

In a similar way, using the elliptic regularity, there is existence and uniqueness of solutions for  the stationary problem, provided the rate of anti-monotonicity of $F$ is not too large. We skip the proof which follows the same lines as in Corollary \ref{corlip}.

\begin{prop}\label{erg} Assume that $p\mapsto H(x,p)$ is a $C^2$  function which satisfies  \rife{Hk} and \rife{lip}, and that $F(x,m)$   satisfies \rife{fK}.

Then there exists $\gamma_0>0$, only depending on $L, \kappa$ and on the functions $F,H$ (through \rife{fK}, \rife{Hk} for some $K$ only depending on $L,\kappa$), such that  if  $F(x,s) + \gamma_0 s$ is nondecreasing, then the stationary ergodic problem \rife{MFGergo} admits a unique solution $(\bar \lambda, \bar u, \bar m)$. 
\end{prop}


\paragraph{(B) Quadratic Hamiltonians and couplings with mild growth}
\ 
\vskip0.3em
Our second example includes the case of superlinear Hamiltonians, having quadratic-like growth in the gradient. Namely, we assume that $H \in C^1(\T \times \R^d)$ is nonnegative and satisfies,  for some $ c_0 > 0$:
\begin{equation}\label{globHass}
\begin{gathered}
c_0^{-1} I_d \le H_{pp}(x, p) \le c_0 \, I_d  , \\
c_0^{-2}|p|^2  \le H_{p}(x, p) \cdot p - H(x,p) , \\
c_0^{-1}|H_{p}(x, p)|^2 \le H_{p}(x, p) \cdot p - H(x,p), \\
|H(x,p) - H(y,p)| \le c_0\, (1+|p|)  , 
\end{gathered}
\end{equation}
for all $x,y \in \T$, $p, q \in \R^d$.

In order to have global bounds, we need here to restrict the growth of the coupling term. Thus, we suppose that $F$ satisfies, for some $\alpha < \frac2d$ and $c_F > 0$,
\begin{equation}\label{fass}
-c_F m^{\alpha} \le F(x,m) \le 0 \qquad \forall m \ge 0.
\end{equation}
Note that it is sufficient that $F$ be bounded from above. Then, one can assume that it is nonpositive by adding a term $Ct$, for suitable $C$, to 
$u$. 
We also assume here that the final datum $G$ is  more regular, i.e.
\be\label{gut2}
G(x, m(T)) = u_T(x) \in C^2(\T).
\ee

\begin{thm}\label{thmquadH} Assume that \eqref{globHass}, \eqref{fass}  and \rife{gut2} are in force. Then, there exist $\M,\U$ depending on $\kappa, \alpha, c_0,c_F,\|m_0\|_\infty$ such that  \eqref{MUT} holds for all classical solutions to \rife{mfgT}. Moreover, assuming in addition the local Lipschitz assumption  \eqref{fK}, there exists another constant $\gamma_0$ (depending also on $\ell_{\M}$ in \eqref{fK})  such that if  
$$
\hbox{$m\,\, \mapsto \,\, F(x,m) + \gamma_0 m \,\,$ is nondecreasing, for 
$ \,\,m \in [0, \mathcal M], x\in \T,$}
$$
then the MFG system \rife{mfgT} has a unique solution.
\end{thm}


We divide the proof of the estimate \eqref{MUT} in several steps. The uniqueness statement is then a straightforward consequence of Theorem \ref{uniq}.
\vskip0.3em
{\it Step 1. Estimates on the oscillation of $u(t)$. } Denote, as usual, ${\rm osc}_{\T} u(t) = \max_{x \in \T} u(x,t) - \min_{x \in \T} u(x,t)$. We claim that there exists $C_0 > 0$ depending on $c_0, c_F, d, \alpha$ such that
\begin{equation}\label{Step1}
{\rm osc}_{\T} u(T-n) \le {\rm osc}_{\T} u_T + 2 C_0\qquad \forall n \in \mathbb N.
\end{equation}
The estimate will be a consequence of the following ``oscillation decay'' 
inequality
\begin{equation}\label{osc0}
{\rm osc}_{\T} u(T-k) \le \frac12 {\rm osc}_{\T} u(T-k+1) +  C_0  \qquad \forall k \in \mathbb N.
\end{equation}
Indeed, given \eqref{osc0}, by induction
\[
{\rm osc}_{\T} u(T-n) \le \frac{1}{2^n} {\rm osc}_{\T} u_T + C_0 \sum_{k=0}^{n-1} \frac{1}{2^k} \le {\rm osc}_{\T} u_T + 2 C_0,
\]
which is \eqref{Step1}. We now turn to \eqref{osc0}. It will be sufficient to prove it for $T-k =0$, being the case $T-k \neq 0$ identical; it suffices indeed to perform a time-shift, which is allowed by the following 
crucial observation: since $\int m(t) = 1$ for all $t$, by \eqref{fass}
\begin{equation}\label{Fok}
\| F(\cdot, m(\cdot, t)) \|_{L^p(\T)} \le c_F, \qquad p = \frac1\alpha > \frac d 2.
\end{equation}
To obtain the oscillation estimate, we will argue by duality. For a review of the so-called adjoint method and its application to mean field games, see \cite{Gomes-book} and the recent developments in \cite{CiGolast}.

 For a smooth probability density $\rho_0 \in C^\infty(\T)$ let $\rho$ be the classical solution to the Fokker-Planck equation
\begin{equation}\label{dualsys} 
\begin{cases}
\rho_t - \Delta \rho - {\rm div}(\rho H_p(x, Du)) = 0, & t \in (0, 1) \\
\rho(x, 0)=\rho_0(x).
\end{cases}
\end{equation}
Then, as $p' = \frac{p}{p-1} < \frac d{d-2}$, by means of Lemma \ref{stimaRho1} there exists $C_\eps$ depending on $\eps, d,\alpha$ (but independent of $\rho_0$) such that
\begin{equation}\label{rhoest}
\| \rho \|_{L^1((0,1); L^{p'}(\T))} \le \eps \int_0^1 \int_{\T} |H_p(x, Du)|^2 \rho + C_\eps \le \eps c_0 \int_0^1 \int_{\T} [H_{p}(x, Du) \cdot Du - H(x,Du)] \rho + C_\eps.
\end{equation}
for $\eps > 0$ that will be chosen below; the second inequality is just a 
consequence of the third assumption \eqref{globHass} on $H$.

We may now add a constant to $u$ so that $\max_\T u(x, 1) = 0$. Note that by the maximum principle, $u(x, t) \le 0$ for all $t \le 1$, $x \in \T$. Using the duality between the equations of $u$ and $\rho$, and estimate \rife{rhoest},  we obtain
\begin{multline}\label{dualo}
\int_0^1 \int_{\T} [H_p(x, Du) Du - H(x, Du)] \rho = \int_\T u(0) \rho_0 - \int_\T u(1) \rho(1) - \int_0^1 \int_{\T} F(x,m) \rho  \\ 
 \le {\rm osc}_{\T} u(1) + \| F(x,m) \|_{L^\infty((0,1); L^{p}(\T))} \| \rho \|_{L^1((0,1); L^{p'}(\T))}   \\
 \le {\rm osc}_{\T} u(1) +  c_0 c_F \eps \int_0^1 \int_{\T} [H_{p}(x, Du) 
\cdot Du - H(x,Du)] \rho + c_F C_\eps. 
\end{multline}
Thus, choosing $\eps = (2c_0 c_F)^{-1}$, we get
\begin{equation}\label{Hp2est}
\int_0^1 \int_{\T} [H_{p}(x, Du) \cdot Du - H(x,Du)] \rho \le 2  \, {\rm osc}_{\T} u(1) + C_2.
\end{equation}

Pick now $x_0, y_0 \in \T$ such that $ {\rm osc}_{\T} u(0) = u(y_0, 0) - u(x_0, 0)$ and let $z = y_0 - x_0$. Setting $\hat \rho (x,t) = \rho(x - (1-t)z, t)$, $\hat \rho$ solves
\[
\begin{cases}
\hat \rho_t - \Delta \hat \rho - {\rm div}\big(\hat \rho H_p(x - (1-t)z, Du(x - (1-t)z,t))\big)- {\rm div}(\hat \rho z ) = 0, & t \in (0, T) \\
\rho(x, 0)=\rho_0(x - z).
\end{cases}
\]
Testing the equation of $\hat \rho$ by $u$ and the equation of $u$ by $\hat \rho$ and integrating by parts we obtain
\begin{multline*}
\int_0^1 \int_{\T} [-H_p(x - (1-t)z, Du(x - (1-t)z,t)) \cdot Du(x,t) - z \cdot Du(x,t) + H(x, Du(x,t))] \hat \rho(x,t)  \\ = - \int_\T u(x,0) \rho_0(x-z) + \int_\T u(1) \rho(1) + \int_0^1 \int_{\T} F(x,m) \hat \rho.
\end{multline*}
Denote for simplicity $y = x + (1-t)z$. After the change of variables $x - (1-t)z \mapsto x$, add \eqref{dualo} to get
\begin{multline*}
\int_0^1 \int_{\T} [-H_p(x, Du(x,t)) \cdot (Du(y,t) - Du(x,t))  + H(y, Du(y,t))  - H(x, Du(x,t))]  \rho(x,t) \\ =  
\int_0^1 \int_{\T}  z \cdot Du(y,t)  \rho(x,t) + 
\int_\T [u(x,0) - u(x+z,0)] \rho_0(x)  + \int_0^1 \int_{\T} F(x,m) (\hat \rho-\rho).
\end{multline*}
Applying now the first and the fourth assumption in \eqref{globHass} to the left-hand side of this resulting inequality yields
\begin{multline*}
\frac{c_0^{-1}}2 \int_0^1 \int_{\T} |Du(y,t) - Du(x,t)|^2 \rho(x,t) +\int_\T [u(x+z,0)-u(x,0)] \rho_0(x)\\  \le
c_0 \int_0^1 \int_{\T}  (1+|Du(y,t)| ) \rho(x,t) + \int_0^1 \int_{\T}  z \cdot Du(y,t)  \rho(x,t) + 
\int_0^1 \int_{\T} F(x,m) (\hat \rho-\rho).
\end{multline*}
Using now Young's and H\"older's inequalities, there exists $c_1$ (only depending on $c_0$)  such that
\begin{multline*}
\int_\T [u(x+z,0)-u(x,0)] \rho_0(x)\\  \le
\frac1{8c_0^2} \int_0^1 \int_{\T}  |Du|^2  \rho + 
2\| F(x,m) \|_{L^\infty((0,1); L^{p}(\T))} \| \rho \|_{L^1((0,1); L^{p'}(\T))} + c_1 \le \\
\frac1{8} \int_0^1 \int_{\T}  [H_{p}(x, Du) \cdot Du - H(x,Du)]   \rho + 
2\| F(x,m) \|_{L^\infty((0,1); L^{p}(\T))} \| \rho \|_{L^1((0,1); L^{p'}(\T))} + c_1,
\end{multline*}
as a consequence of the second assumption in \eqref{globHass}. Finally, we plug in \eqref{rhoest} and \eqref{Hp2est}, and by an appropriate choice 
of $\eps$ small we obtain
\[
\int_\T [u(x+z,0)-u(x,0)] \rho_0(x)  \le \frac12 {\rm osc}_{\T} u(1) + C_3,
\]
for some $C_3$ depending on $c_0, c_F, \alpha, d$. Choosing now a sequence of $\rho_0$ converging (weak-*) to $\delta_{x_0}$ we obtain the desired 
estimate
\[
{\rm osc}_{\T} u(0) = u(y_0, 0) - u(x_0, 0) \le \frac12 {\rm osc}_{\T} u(1) + C_3.
\]

\medskip

{\it Step 2. Estimates on a $C^\beta$-norm of $u$. } We claim that there exists $\beta \in (0,1)$ and $C > 0$ depending on $c, c_F, \alpha, d$ such that for all $n \in \mathbb N$
\be\label{uTn}
\|u(\cdot, t) - \max_{\T} u(T-n) \|_{C^\beta(\T)} \le C \qquad \text{for all $t \in [T-(n+2), T-(n+1)]$,}
\ee
and the inequality can be extended up to $t = T$ when $n=0$. 

\vskip0.2em
First, $L^\infty$-bounds on $z(x,t) := u(x, t) - \max_{\T} u(T-n)$ can be obtained by duality as in Step 1 (and the argument is even simpler). Note that $\|z(T-n) \|_{\infty} \le C$ by \eqref{Step1}, independently on $T, n$. We then proceed assuming without loss of generality that $[T-(n+2), T-n] = [-1,1]$. Since $z(1) \le 0$ and $F \le 0$, we have $z \le 0$ on $\T \times (-\infty,1]$ by the maximum principle.
Let now $\rho$ be as in \eqref{dualsys}. Arguing as before by duality (see equations \eqref{rhoest}-\eqref{Hp2est}), there exists a constant $\overline C$ depending on $c_0, c_F, \|z(1)\|_\infty$ (but not depending on $\rho_0$) such that
\[
\left|\int_0^1 \int_{\T} F(x,m) \rho\right|  \le \overline C,
\]
and therefore
\begin{multline*}
\int_\T |z(0)| \rho_0 = -\int_\T z(0) \rho_0 =\\
  -  \int_0^1 \int_{\T} [H_p(x, Dz) Dz - H(x, Dz)] \rho - \int_\T z(1) \rho(1) - \int_0^1 \int_{\T} F(x,m) \rho \le \|z(1)\|_\infty + \overline C.
\end{multline*}
Varying $\rho_0$ yields a bound on $\|z(\tau)\|_\infty$ for $\tau = 0$. 
Then, varying $\tau \in [-1,1)$ in the initial condition $\rho(x, \tau)=\rho_0(x)$ for $\rho$ allows to extend such $L^\infty$-bounds for $z$ to the whole cylinder $\T \times [-1,1]$.

Once sup-bounds on $u(\cdot, t) - \max_{\T} u(T-n)$ are established, using the uniform integrability of $F(m)$ in \eqref{Fok}, a control on a H\"older semi-norm follows by standard results for quasi-linear parabolic equations with quadratic growth in the gradient, see e.g. \cite[Theorem V.1.1]{LSU}.

\medskip

{\it Step 3. Estimates on the $L^q(L^p)$-norm of $|Du|^2$. } Let $q > 1$. 
We claim that there exists $C > 0$ depending on $c_0, c_F, \alpha, d, q$ such that for all $n \in \mathbb N$
\begin{equation}\label{Step3}
\|Du \|_{L^{2q}(\, (T-(n+1), T-n)\, ; L^{2p}(\T))} \le C.
\end{equation}
We prove the inequality in the case $(T-(n+1), T-n) = (0,1)$ and $T \ge 
2$; constants below will not depend on $T$ nor $n$, so the validity of \eqref{Step3} for $t \in [0,T-1]$ will be a straightforward consequence. Some comments regarding  the interval $t \in [T-1,T]$, that is for $t$ close to the time-horizon will be made below.

First, $\tilde u(x,t) = (2-t)[u(x,t) - \max_{\T} u(x,2)]$ solves $\tilde u(x,2) = 0$ and
\[
-\tilde u_t -  \Delta \tilde u = - (2-t)H\left(x, \frac{D \tilde u}{2-t}\right) + (2-t) F(x, m(t))
\]
Therefore, for any $q > 1$, by maximal $L^q-L^p$ regularity for linear parabolic equations  (see e.g. \cite{HP}), there exists $C_q > 0$ depending 
on $q, p, d$ such that
\[
\int_0^2 \|\tilde u(t)\|^q_{W^{2,p}(\T)} \le C_q \int_0^2  \Big\|(2-t) H\left(x, \frac{D \tilde u}{2-t}\right)\Big\|^q_{L^{p}(\T)} + \|F\|^q_{L^p(\T)} dt
\]
Since $H(x,p)$ has quadratic growth in the $p$-variable, we may adjust $C_q$, and use \eqref{Fok} to obtain
\begin{equation}\label{eq245}
\int_0^2 \|\tilde u(t)\|^q_{W^{2,p}(\T)} \le C'_q  \left( \int_0^2  \|D\tilde u(t)\|^{2q}_{L^{2p}(\T)}\frac{dt}{(2-t)^q} + 2c_F^q \right).
\end{equation}
We now recall the following Gagliardo-Nirenberg type inequality
\begin{equation}\label{GN}
\|D\tilde u(t)\|_{L^{2p}(\T)}\leq C\|\tilde u(t)\|_{W^{2,p}(\T)}^\theta\|\tilde u(t)\|_{C^{\beta}(\T)}^{1-\theta},
\end{equation}
which holds for $\theta\in\left[\frac{1-\beta}{2-\beta},1\right)$ and
\begin{equation*}
\frac{1}{2p}=\frac{1}{d}+\theta\left(\frac{1}{p}-\frac{2}{d}\right)-(1-\theta)\frac{\beta}{d}.
\end{equation*}
Then we pick $\beta > 0$ as in the previous Step 2. Note that since $p > \frac d2$, we have $\theta < 1/2$, and by Step 2
\[
\frac{ \|D\tilde u(t)\|^{2q}_{L^{2p}(\T)} }{(2-t)^{q}} \leq \frac{C \|\tilde u(t)\|_{W^{2,p}(\T)}^{2q \theta} \|\tilde u(t)\|_{C^{\beta}(\T)}^{q(2-2\theta)} }{(2-t)^{q}}  \le C' (2-t)^{q(1-2\theta)} \|\tilde u(t)\|_{W^{2,p}(\T)}^{2q\theta}.
\]
Thus, plugging the previous inequality into \eqref{eq245} yields
\[
\int_0^2 \|\tilde u(t)\|^q_{W^{2,p}(\T)} dt \le C_4  \left( \int_0^2  \|\tilde u(t)\|_{W^{2,p}(\T)}^{2q\theta} dt + 1 \right)
\]
for all $t \in (0,2)$. Since $2q\theta < q$, an estimate on $L^q((0, 2); W^{2,p}(\T))$ for $\tilde u$ follows. In turn, back to \eqref{GN}, this gives bounds in $L^{2q}((0, 2); L^{2p}(\T))$ for $D \tilde u$. Finally, claimed bounds in $L^{2q}((0, 1); L^{2p}(\T))$ for $D u$ are straightforward.

In the interval $[T-1,T]$ there is no need to localize in time with the term $(2-t)$ and normalize the sup-norm, i.e. it is sufficient to perform the very same argument with $\tilde u(x,t) = u(x,t)$ (and use that $u(T)$ is $C^2$).

\medskip

{\it Step 4. Estimates on the sup-norm of $Du$ and $m$. } By the assumptions on $H_p$, which has linear growth in $|p|$, the previous estimate \eqref{Step3} reads
\[
\|H_p(x, Du) \|_{L^{2q}(\, (T-(n+1), T-n)\, ; L^{2p}(\T))} \le C.
\]
for any $q > 1$ and for some $p > d/2$. Hence,  $m$ solves a linear equation in divergence form with drift $H_p(x, Du)$, that in turn satisfy the previous integrability condition. Since $\|m(t)\|_{\elle1}=1$ for  all $t$,   it is standard the existence of $\mathcal M$ (independent of $T$) such that
\[
\max_{[0, T] \times \T} m \le \mathcal M,
\]
see e.g. \cite[Theorem III.7.1]{LSU}. Hence,  now we have $\max_{[0, T] \times \T} |F(x,m)| \le c_F \mathcal M^\alpha$. Then, reasoning as in Step 
2 in any interval 
$[T-(n+2), T-(n+1)]$, where we use that $u(t)-u(T-n)$ is bounded
uniformly  (see \rife{uTn}),  we can apply   \cite[Theorem V.3.1]{LSU} in 
order to get a  bound for  $Du$ at time  $T-(n+2)$. Since this bound  is independent  of $T$, and   thanks  to \rife{gut2}, we conclude that  a uniform  bound holds up to $t=T$:
\[
\max_{[0, T] \times \T} |D u| \le \mathcal U\,.
\]
Hence, \eqref{MUT} is proved.
\qed

\bigskip

A similar result also holds for the stationary ergodic problem, as well.

\begin{thm} Assume that $H$ satisfies \rife{globHass} and $F$ satisfies \rife{fass}. Then there exists  a  solution $(\bar\lambda, \bar m, \bar u)$ of the ergodic problem \rife{MFGergo} such that $\bar m\in \elle\infty$, $\bar u \in W^{1,\infty}(\T)$. 

Moreover,  there exist $\bar \M, \bar \U$ (only depending on $\kappa, \alpha$ and the constants $c_0, c_F$) such that any solution of \rife{MFGergo} satisfies
$$
\max_{\T} \bar m \le \bar \Mcal, \qquad \max_{\T} |D \bar u| \le \bar \Ucal\,.
$$
Finally, there exists $\gamma_0>0$ (only depending on $\kappa, \alpha$ and the constants $c_0, c_F$) such that, if  $F(x,s)+ \gamma_0 s$ is nondecreasing, then the solution 
$(\bar\lambda, \bar m, \bar u)$ is unique.
\end{thm}

\proof We first prove the second assertion, namely the a priori estimate. 
Let $(\bar\lambda, \bar m, \bar u)$ be any solution of \rife{MFGergo}. 
We start with bounds on $\bar \lambda$ (that are somehow related to oscillation estimates in the previous part). First, $\bar \lambda \le -\max_\T 
H(\cdot, 0)$ (it just suffices to evaluate the equation for $\bar u$ at a 
maximum point of $\bar u$). Then, we use an estimate  in \cite[Proposition 2.3]{CCPDE}: since $\alpha < \frac 2d$, there exists $C > 0$ and $\sigma < 1$ (depending on $d$) such that
\[
\int_{\T} \bar m^{\alpha+1} \le C \left(\int_{\T} |H_p(x, D \bar u)|^2 \bar m + 1 \right)^{\sigma}
\] 
Hence, testing the equation  of $\bar m$ by $\bar u$ and the equation of  
$\bar u$ by $\bar m$ and integrating by parts we obtain
\begin{multline*}
\bar \lambda = \int_{\T} [H_p(x, D\bar u) D\bar u - H(x, D\bar u)]\bar m  + \int_{\T} F(x,\bar m) \bar m \label{dualo}  \\ 
\ge c_0^{-1}\int_{\T} |H_p(x, D \bar u)|^2 \bar m - c_F C \left(\int_{\T} 
|H_p(x, D \bar u)|^2 \bar m + 1 \right)^{\sigma},
\end{multline*}
which is clearly bounded from below by a positive constant depending on $c_F, c_0, C, \sigma$. Therefore, $\bar \lambda$ is bounded only in terms of $\kappa, c_0, c_F, \alpha$.

Moreover, since $\int \bar m = 1$, by \eqref{fass}
\[
\| F(\cdot, \bar m(\cdot)) \|_{L^p(\T)} \le c_F, \qquad p = \frac1\alpha > \frac d 2.
\]
Thus, we obtain bounds on $|- \Delta \bar u + H(x, D\bar u)|$ in $L^p(\T)$. Note that a straightforward control of the $L^2$-norm of $D \bar u$ comes from integration on $\T$ of the HJB equation for $\bar u$. Therefore, 
maximal regularity \cite[Theorem 1.1]{CG} results yield bounds on $H(x, D\bar u)$ in $L^p(\T)$, and then on $H_p(x, D\bar u)$ in $L^{2p}(\T)$, $p > d/2$. The existence of  $\bar \Mcal$ such that $\max_{\T} \bar m \le \bar \Mcal$ is then classical (e.g. \cite[Section 3.13]{LadyElliptic}). Being the right-hand side of the HJB equation bounded in sup-norm by $c_F \bar \Mcal^\alpha$, it follows  (e.g. \cite[Section 4.3]{LadyElliptic}) that  $\max_{\T} |D \bar u| \le \bar \Ucal$ for some $\bar \U$ only depending on $\bar \M$ and $\bar \lambda$. This concludes the a priori estimate. 

Note that the above procedure holds also if one replaces $F$ by its truncation $ = F(x, \min\{m, \Mcal\})$, since it depends only on the upper bound \eqref{fass} on $|F|$. Therefore, one may consider a classical solution of the system
\be\label{fixpoi}
\begin{cases}
\bar  \lambda -\Delta u  + H(x,Du)=\overline{F}(x,m) & \qquad {\rm in }\; \T  \\
  -\Delta  m  -{\rm div} \left( m\,  H_p(x,Du)\right)=0 & \qquad {\rm in }\; \T 
\\
  { \into m=1 \; , \; \into u=0 } & 
\end{cases}
\ee
which exists, e.g, by results in \cite{CCPDE} ($\overline{F}$ is globally 
bounded). Since $m \le \Mcal$, then $\overline F (x,m)=F (x,m)$, and we 
have the existence statement of a solution to the original ergodic problem.

In view of  the  a priori bounds found above, the uniqueness statement is 
a direct consequence of Theorem \ref{uniq-erg}. 
\qed 


\begin{rem}\label{remgrow} {\it On the growth assumption $\alpha < \frac 2d$. } We have seen in this section that a condition of mild growth of $F(x, \cdot)$ guarantees the existence of solutions to the MFG systems (both the ergodic and the evolutive one). It is worth noting that the existence of a triple $(\bar u, \bar m, \bar \lambda)$ to the ergodic MFG system 
has been established under the weaker growth assumption $\alpha < \frac 2{d-2}$ and additional {\it smallness} constraints on $c_F$, while non-existence of solutions might even arise in the regime $\alpha > \frac 2{d-2}$, see \cite{CCPDE}. Concerning the evolutive MFG system \eqref{mfgT}, existence of solutions for arbitrarily large time horizon $T$ may  fail already when $\alpha > \frac 2d$ in general \cite{CGhi} (so the study of the 
long time behavior becomes much more delicate, being the MFG system even more sensitive to the data). Still, smallness assumptions on $c_F$ are sufficient to recover existence for all $T > 0$, as described in \cite{CGhi}. These assumptions may   then guarantee uniqueness also, and the turnpike property, but such an analysis is a bit beyond the scopes of this work.

\end{rem}

\section{The exponential turnpike estimate}

In this section we prove that, if the anti-monotonicity of the coupling $F(x,m)$ is sufficiently small, then we can prove the existence of solutions of
\rife{mfgT} satisfying the turnpike property.
The strategy we adopt   follows \cite[Section 1.3.6]{CP-CIME}, through the construction of a solution via a fixed point in a suitable weighted space.


\begin{thm}\label{longtime}  Let $m_0\in \PT$. Assume that $F(x,m), G(x,m)$  satisfy
\rife{fK}, \rife{gK},  and  that $ H(x,p)$  satisfies \rife{Hk} and \rife{lip}.

Then there exists $\gamma>0$  only depending on $L, \kappa$ (and on the functions $F,H$), such that  if  $F(x,s) + \gamma s$ is nondecreasing then 
any solution  $(u^T,m^T)$ of problem \rife{mfgT}  satisfies
\be\label{turnp2}
\| m^T(t)-\bar m\|_\infty + \| Du^T(t)-D\bar u\|_\infty \leq M (e^{-\omega t}+ e^{-\omega (T-t)}) \qquad \forall t\in (1,T-1)\,,
\ee
for some $\omega, M>0$ (independent of $T$), where $(\bar \lambda, \bar u, \bar m)$ is given by Proposition \ref{erg}.
\end{thm}

\begin{rem}  According to the Theorem, there is a threshold of anti-monotonicity, namely there exists some  $\hat \gamma>0$ such that if $F(x,s)+ \gamma s$ is nondecreasing with $\gamma<\hat \gamma$, then any solution enjoys the turnpike estimate. Of course, we have $\hat \gamma \leq \gamma_0$ given by Proposition \ref{erg} (indeed, a unique stationary state is used here). This value $\hat \gamma$ depends on $F,H$ through the constant $L$ in \rife{lip} and through assumptions \rife{fK}, \rife{Hk}, in the sense that it depends on the constants $c_K, \ell_K, \alpha_K$ for a value  of $K$ only depending on $L,\kappa$. No special effort is devoted, in the proof below, to catch  a refined estimate of $\hat \gamma$; as it will be clear, several arguments will need $\gamma$ to be smaller than generic constants appearing in the global (in time) estimates of the solutions.
\end{rem}

The proof of Theorem \ref{longtime} will rely on the application of Schaefer's fixed point theorem (\cite[Thm 11.3]{GT}). The key-point is given in the following lemma which contains an a priori estimate on a sort of linearization of the mean field game system. We recall that, for any $v\in \elle2$, we denote $\lg v\rg= \into v $ and $\tilde v= v- \lg v\rg$.

\begin{lem}\label{apriori}
Let $h(x,p)$ be differentiable with respect to $p$, and assume that $h(x,p), h_p(x,p)$, $f(x, s)$ and $B(x,p)$ are all  Carath\'eodory  functions  which satisfy  the following growth conditions for some constants $\ell_0, C_0, C_1, C_2$ and  for every $s\in \R$,  $x\in \T$ and $p\in \R^d$ such that $|p|\leq K$:
 \be\label{hT}
 h(x,0)=0\,,\quad |h_p(x,p)| \leq \ell_0\,,
 \ee
 \be\label{fT}
 f(x,s)s\geq -\gamma \,s^2\,,\quad |f(x,s)| \leq C_0\,,\quad |f(x,s)|\leq C_1 \, |s|
 \ee
\be\label{BT}
B(x,p)\cdot p \geq C_2^{-1} |p|^2\,, \qquad |B(x,p)|\leq C_2\, |p| \,.
 \ee 
For $\sigma\in [0,1]$,  $\mu_0\in \elle2$,  with $\into \mu_0=0$, and $v_T\in \elle2$,  let $(\mu,v)$ be a solution of  the system
 \be\label{ls}
 \begin{cases}
-\partial_t v - \Delta v + h(x, Dv)  = f(x,  \mu)   &  t\in (0,T),
\\
v(T)= v_T& 
\\
\partial_t \mu  - \Delta \mu- \dive(\mu\, h_p(x, Dv)) = \sigma\, \dive(B(x,Dv)) &  \,t\in (0,T),
\\
\mu(0)= \sigma \, \mu_0
& \end{cases}
\ee
where we  assume that, for any $(t,x)\in Q_T$,  we have $|Dv(t,x)|\leq K$ and 
\be\label{summa}
\begin{split}
\sigma B(x,p)\cdot p -   \mu(t,x) (h(x,p)-h_p(x,p)\cdot p)      \geq \sigma\,c_0\,  |p|^2\qquad \forall (t,x)\in Q_T, \forall p\,:|p|\leq K\,, 
\end{split}
\ee
for some $c_0, K>0$. 
 
 Then there exist constants $\gamma_0, \omega, c >0$ (independent of $\sigma, v,\mu$) such that, if $\gamma\leq \gamma_0$ ($\gamma$ is in \rife{fT}), 
 then  $(\mu,v)$   satisfies
\be\label{stimal2}
\|\mu(t)\|_2 + \|\tilde v(t)\|_2 \leq  c\, [\|\mu_0\|_2+\|\tilde  v_T\|_2] \left( e^{-\omega t}+ e^{-\omega (T-t)}\right) \qquad \forall t\in (0,T)\,,
\ee
 where $\tilde v(t)= v(t)- \lg v(t) \rg$.  The constants $\gamma_0, \omega, c$ only depend on  $\kappa, \ell_0, C_1, C_2, c_0$. 
\end{lem}

\begin{proof} 
For $T>0$, $\sigma\in [0,1]$, $\mu_0\in \elle2$ with $\into \mu_0=0$,  and $v_T\in L^2(\T)$,  let  $(\mu,v)$ be  the  solution of system \rife{ls}.   
We first prove that there exists a constant $c$, independent of $\sigma,T,\mu_0, v_T$, such that 
\be\label{apri}
\|\mu(t)\|_2+ \|\tilde v(t)\|_2 \leq c (\|\mu_0\|_2+ \|\tilde v_T\|_{2}) \,.
\ee
To start with, we observe that,  due to \rife{summa} and \rife{fT}, $(\mu,v)$ satisfies
\be\label{ddtT}
\begin{split}
-\frac{d}{dt} \into \mu(t)v(t) & = \into f(x,  \mu) \mu + \sigma \into B(x,Dv)Dv 
 -   \into \mu (h(x,Dv)-h_p(x,Dv)\cdot Dv) 
\\
& \geq \sigma \, c_0\into |Dv|^2 - \gamma  \into | \mu|^2\,.
\end{split}
\ee
Integrating and  using Lemma \ref{stimaRhoF} we get, for a constant $C$ only depending on $\kappa, \ell_0$ (given by \rife{hT}):
\begin{align*}
 \sigma \, c_0 \int_0^T \int_\T |Dv|^2  &  \le (v (0), \mu_0) - (v_T, \mu(T) ) + \gamma  \int_0^T \|\mu(t)\|_2^2 dt \\ 
 & \le (v (0), \mu_0) - (v_T, \mu(T) ) + C\gamma  
 \|\mu_0\|_2^2 +   C\, \gamma \sigma^2 \int_0^T \int_\T |B(x,Dv)|^2 
 \\  & \le (v (0), \mu_0) - (v_T, \mu(T) ) + C\gamma  
 \|\mu_0\|_2^2 +   C\, C_2^2\gamma \sigma \int_0^T \int_\T |Dv|^2
\end{align*}
where we used \rife{BT} and $\sigma\leq 1$.
If $\gamma C\, C_2^2< c_0$ we deduce the bound
\be\label{axel1} 
\sigma \int_0^T \int_\T   |  Dv |^2 \le c\,    \left\{\|\mu_0\|_2   \big[ \|\tilde v (0)\|_2 + \|\mu_0\|_2 \big]+   \| \tilde v_T\|_2 \|\mu(T)\|_2\right\}.
\ee
Hereafter, we denote by $c$ possibly different constants, depending on $\kappa, \ell_0, C_1, C_2, c_0$, which may vary from line to line. Those constants are independent of $\sigma, T, \mu_0, v_T$.

We deduce from \rife{axel1}, using again    Lemma \ref{stimaRhoF}  and \rife{BT},
\begin{align*}
\sup_{t\in [0,T]}\,  \, \|\mu (t)\|_2^2   &  \leq  C  \left( \|\mu_0\|_2^2 +\sigma^2 C_2^2 \int_0^T \int_\T   |  Dv  |^2 \right)  \\
& \leq     c\,    \left\{\|\mu_0\|_2   [ \|\tilde v (0)\|_2 + \|\mu_0\|_2 ]+   \| \tilde v_T\|_2 \|\mu(T)\|_2\right\}
\end{align*}
which implies
\be\label{m2}
\sup_{[0,T]} \|\mu(t)\|_2^2 \leq c\,  [\|\mu_0\|_2^2+ \| \tilde v_T\|_2^2] +c\, \|\mu_0\|_2 \|\tilde v(0)\|_2\,. 
\ee
Since the Hamilton-Jacobi equation implies (using Lemma \ref{lem72} and \rife{fT})
\begin{align*}
\|\tilde v(0)\|_2 &  \leq C \, e^{-\nu T}\| \tilde v_T\|_2 + C\, C_1 \int_0^Te^{-\nu s}\| \mu(s)\|_2 ds
 \leq  C\, \| \tilde v_T\|_2 + \frac {CC_1}\nu \, \sup_{[0,T]} \|\mu(t)\|_2 
\end{align*}
coming back to \rife{m2} we deduce (for possibly different $c$)
$$
\sup_{[0,T]} \|\mu(t)\|_2^2 \leq c\,  [\|\mu_0\|_2^2+ \|\tilde v_T\|_2^2]\,.
$$
A similar estimate follows  for $\sup_{[0,T]}\|\tilde v(t)\|_2$,  
using  again Lemma \ref{lem72}.  This allows us to conclude estimate   \rife{apri}. 
In addition, the inequalities above also show that 
\be\label{intdv}
\sigma \int_0^T \int_\T   |  Dv |^2 dt + \int_0^T \int_\T   |  \mu |^2 dt\leq c \,  [\|\mu_0\|_2^2+ \|\tilde v_T\|_2^2]\,.
\ee
We claim now that this implies the existence of $\tau$ such that, for every $T>2\tau$ we have 
\be\label{tnot}
\|\mu(t)\|_2+ \|\tilde v(t)\|_2 \leq \frac12 [\|\mu_0\|_2 + \|\tilde  v_T\|_2 ] \qquad \forall t\in [\tau,T-\tau]\,.
\ee
In fact, using \rife{intdv}, we know that there exist points $\xi_\tau \in [0,\tau/2] $ and $\eta_\tau\in [T- \tau/2, T]$ such that
\be\label{muxitau}
 \|\mu(\xi_\tau)\|_2^2   \leq \frac {2c}{\tau} M^2\,,\quad  \|\mu(\eta_\tau)\|_2^2   \leq \frac {2c}{\tau}M^2\,,\qquad \hbox{where $M^2=\|\mu_0\|_2^2 + \|\tilde v_T\|_2^2  $.}
 \ee 
Estimating once more $\mu$ through Lemma \ref{stimaRhoF}, and then using  \rife{ddtT} (integrated in the interval $(\xi_\tau, \eta_\tau)$), thanks to \rife{muxitau} we get 
\begin{align*}
\int_{\xi_\tau}^{\eta_\tau} \|\mu(t)\|_2^2 dt  &  \leq C   \|\mu(\xi_\tau)\|_2^2 + C\,C_2^2 \sigma^2 \int_{\xi_\tau}^{\eta_\tau} \into |Dv|^2 
\\
& \leq 
 \frac {2c\,C}{\tau} M^2 +  \frac{C\,C_2^2}{c_0} \left\{  \gamma \int_{\xi_\tau}^{\eta_\tau} \|\mu(t)\|_2^2 +    \into \mu(\xi_\tau)v(\xi_\tau)- \into \mu(\eta_\tau)v(\eta_\tau) \right\}
\end{align*}
Using the global bound for $\|\tilde v(t)\|_2$ and \rife{muxitau} we estimate last two terms. In the end, choosing $\gamma$ sufficiently small we deduce
\be\label{ml2}
\int_{\xi_\tau}^{\eta_\tau} \|\mu(t)\|_2^2 dt  \leq M^2 ( \frac c{\tau} + \frac c{\sqrt \tau})
\ee
and in turn what we estimate in between  gives (as we said before, for possibly different $c$)
\be\label{vl2}
\sigma^2 \int_{\xi_\tau}^{\eta_\tau} \into |Dv|^2 \leq M^2 ( \frac c{\tau} + \frac c{\sqrt \tau})\,.
\ee
Using once more Lemma  \ref{stimaRhoF} we have, for all $t\in (\xi_\tau, \eta_\tau)$
$$
\|\mu(t)\|_2^2  \leq C  \left( \|\mu(\xi_\tau)\|_2^2+ \sigma^2\, C_2^2  \int_{\xi_\tau}^{\eta_\tau} \into |Dv|^2\right)
$$
and so \rife{muxitau} and \rife{vl2} yield
\be\label{muxitau2}
\|\mu(t)\|_2^2   \leq   M^2 ( \frac c{\tau} + \frac c{\sqrt \tau})\qquad \forall t\in (\xi_\tau, \eta_\tau)\,.
\ee
Similarly we estimate $\tilde v$, using Lemma \ref{lem72} and \rife{fT}. For $t \in (\xi_\tau, \eta_\tau)$ we have
\be\label{vnot}
\begin{split}
\|\tilde v(t)\|_2  & \leq C\, e^{-\nu (\eta_\tau -t) } \|\tilde v(\eta_\tau)\|_2  + C\, C_1  \int_{t}^{\eta_\tau}  e^{-\nu(s-t)}\|\mu(s)\|_2 ds
\\
& \leq c\,  e^{-\nu (\eta_\tau -t) } M  +      M  ( \frac c{\sqrt \tau} + \frac c{\tau^{\frac14}})\,.
\end{split}
\ee
where we used \rife{apri} and \rife{muxitau2}.   For $t \in [\tau, T-\tau]$ we have $\eta_\tau -t \geq \frac \tau 2$, hence we get
$$
\|\tilde v(t)\|_2^2  \leq c\, e^{-\nu \tau/2 } M  +     M  ( \frac c{\sqrt \tau} + \frac c{\tau^{\frac14}})\,.
$$
This inequality, together with \rife{muxitau2}, imply \rife{tnot} for a convenient choice of $\tau$. 

Finally, by iteration of \rife{tnot}, we deduce that there exist  $\omega>0$ and $c>0$ such that
$$
\|\mu(t)\|_2 + \|\tilde v(t)\|_2 \leq  c\, [\|\mu_0\|_2+\|\tilde  v_T\|_2]  \left( e^{-\omega t}+ e^{-\omega (T-t)}\right) \qquad \forall t\in (0,T)\,,
$$
which is \rife{stimal2}.
\end{proof}

Now we complete the proof of Theorem \ref{longtime}.
\vskip0.4em

{\bf Proof of Theorem \ref{longtime}.}\quad 

{\it Step 1} (definition of the fixed  point mapping) We will first prove the result assuming  $m_0\in \PT \cap C^{0,\alpha}(\T)$. We  set  $X= C^0([0,T]; L^2_0(\T))$, where $L^2_0(\T)$ denotes the subspace of $\elle2$ made of functions with zero average. Then  we introduce the following norm in $X$:
 $$
||| u |||_X:= \sup_{[0,T]} \left(\frac{\|u(t)\|_{L^2(\T)} }{e^{-\omega t}+ e^{-\omega(T-t)}}\right)
 $$
where $\omega>0$ will be chosen later.  We have that $(X, ||| u |||_X)$ is a Banach space and $|||\cdot |||$ is equivalent to the standard norm $\|u\|=\sup_{[0,T]}  \|u(t)\|_{L^2(\T)}  $.
 
By Proposition \ref{erg}, there exists some $\gamma_0$ such that, if $F(x,s)+ \gamma_0 s$  is nondecreasing, then there exists a unique $(\bar \lambda, \bar m, \bar u)$ solution of the stationary ergodic problem \rife{MFGergo}.

We define the operator $\Phi$ on $X$ as follows: given $\mu\in X$, let $(v,\rho)$ be the solution to the system
\be\label{sys}
\begin{cases}
-v_t - \kappa \Delta v + H(x,D\bar u + Dv) - H(x,D\bar u)= F(x,\bar m + \mu) - F(x,\bar m) & 
\\
v(T)= G(x, \bar m + \mu(T))-\bar u & 
\\
\rho_t - \kappa \Delta \rho- \dive(\rho\, H_p(x,D\bar u + Dv)) = -\dive(\bar m \left[ H_p(x,D\bar u + Dv)- H_p(x, D\bar u)\right]) & 
\\
\rho(0)= m_0-\bar m\,.
& \end{cases}
\ee
Then we set $\rho= \Phi\mu$. We observe that the existence and uniqueness of $(v,\rho)$ is well known because $H$ satisfies \rife{lip}. We point out that   $\mu$  is a  fixed point of $\Phi$ if and only if $m:= \bar m + \mu$ and $u:= \bar u + \bar \lambda(T-t)+ v$ yield  a solution of \rife{mfgT}. 

{\it Step 2}. (a priori estimates and existence of  a  fixed point) We first observe that, due to \rife{lip}, there exists $  \mathcal R$, only depending on $L, \kappa$ and $\|m_0\|_\infty, \|\bar m\|_\infty$, such that 
\be\label{rhoR}
\|\rho\|_\infty \leq { \mathcal R}\qquad \forall \rho\in {\rm Range} (T)\,.
\ee
Therefore, up to replacing $ \mu$ with $\min(\mu, \mathcal R)$ in the first equation of \rife{sys}, we can assume that $F(x,\cdot)$ is globally bounded and Lipschitz, thanks to \rife{fK} used with $K= \|\bar m\|_\infty + \mathcal R$.  Using again \rife{lip} and the global bound of the right-hand side, we deduce (e.g. by Lemma \ref{lem72}) that   a global bound holds for $\|\tilde v(t)\|_2$. Hence, by (local) regularizing effect of parabolic equations,   there exists a constant $K>0$ (only depending on $L,\kappa$ and ${ \mathcal R}$) such that 
\be\label{lipv}
\| Dv(t)\|_\infty \leq K \qquad \forall t\leq T-1\,,\qquad \forall \mu \in X\,.
\ee
The continuity of the operator $\Phi$ in $C^0([0,T]; L^2(\T))$ (hence in $X$) is a routine stability argument for parabolic equations, due to the Lipschitz character of $H$ and the boundedness of $F$. In addition, since $\bar m, m_0\in C^{0,\alpha}$, by standard regularity results  (see \cite[Chapter V, Thms 1.1 and 2.1]{LSU}) we have that the range of $\Phi$ is bounded in $C^{\alpha/2,\alpha}(Q_T)$, in particular   its closure is compact in $X$. 
Thus, $\Phi$ is a compact and continuous operator. In order to apply Schaefer 's fixed point theorem, we 
 are left to prove  the following claim: there exists a constant $M>0$ such that 
\be\label{claim}
||| \mu|||\leq M \quad \hbox{ for every $\mu\in X$ and every $\sigma\in [0,1]$ such that $\mu= \sigma \Phi(\mu)$.}
\ee 
Now we observe that  the estimate  \rife{claim} follows from Lemma \ref{apriori} (which will be applied in $(0,T-1)$ due to \rife{lipv}). Indeed, if  $\mu= \sigma \Phi(\mu)$,  then $(\mu,v)$  is a solution to 
\rife{ls} with $\mu_0= m_0-\bar m$, $v_{T-1}= v(T-1,x)$, and where $h(x,p), f(x,s), B(x,p)$ are defined by 
\begin{align*}
& h(x,p):= H(x,D\bar u(x) + p) - H(x,D\bar u(x))
\\
&  f(x,\mu):= F(x,\bar m(x) + \mu) - F(x,\bar m(x))
\\
& 
B(x,p):= \bar m(x) \left[ H_p(x,D\bar u(x) + p)- H_p(x, D\bar u(x))\right]\,.
\end{align*}
Using \rife{rhoR} and \rife{lipv}, together with \rife{fK} and \rife{Hk}, the functions  $h(x,p), f(x,s), B(x,p)$ satisfy the conditions \rife{hT}-\rife{BT}, 
where we also used that $F(x,s)+ \gamma s$ is nondecreasing.
 
In addition, since $\mu= \sigma \Phi(\mu)$ implies $\mu(t,x)\geq -\sigma \bar m(x)$, we also have, for some constant $c_0$,
\begin{align*}
\sigma B(x,p)\cdot p -   &  \mu(t,x) (h(x,p)-h_p(x,p)\cdot p)      \geq \sigma B(x,p)\cdot p- \sigma  \bar m(x) (h_p(x,p)\cdot p-h(x,p))
\\
& = \sigma \bar m(x) \left[ - H_p(x, D\bar u(x))\cdot p+ H(x,D\bar u(x) + p) - H(x,D\bar u(x))\right]
\\
& \geq \sigma\,c_0\,  |p|^2\qquad \forall(t,x)\in Q_T, \forall p\in \R^d\,: |p|\leq K\,,
\end{align*}
where we used the local uniform convexity of $H$  and that $\bar m>0$. Therefore, condition \rife{summa} holds too. 
Applying Lemma \ref{apriori} we deduce that there exists a constant $c$ (independent of $\sigma, T$) such that
$$
\|\mu(t)\|_2 + \|\tilde v(t)\|_2 \leq  c\, [\|\mu_0\|_2+\|\tilde  v(T-1)\|_2] \left( e^{-\omega t}+ e^{-\omega (T-t)}\right) \qquad \forall t\in (0,T-1)\,.
$$
Since $\| \tilde v\|_2$  is uniformly bounded, this yields 
$$
\|\mu(t)\|_2   \leq  M \left( e^{-\omega t}+ e^{-\omega (T-t)}\right) \qquad \forall t\in (0,T-1)\,.
$$
The estimate extends to  $(0,T)$ because $\mu$ is also uniformly bounded, thus we proved that $||| \mu|||\leq M$ for some $M$ independent of $\sigma,T$.
This proves \rife{claim} and concludes the fixed point argument. Eventually, one can upgrade the estimate in $(1,T-1)$ to the $L^\infty$-norm of $\mu(t)$ and $Dv(t)$ by the  regularizing effect in the two equations. This latter argument is already developed e.g. in \cite{P-MTA}. 

{\it Step 3.} (conclusion of the general case) Assume now that $m_0 \in \PT$ and $(u,m)$ is any solution of \rife{mfgT}. Due to \rife{lip}, by standard regularizing effect we have that $m(t) \in C^{0,\alpha}(\T)$ for some $\alpha\in (0,1)$, and that $\|m(t)\|_\infty\leq c\, (t-t_0)^{-{d/2}}$, for every $t\in (t_0,t_0+1)$ and every $t_0>0$. Therefore, $m$ is uniformly bounded (globally in time) in the interval $(1,T)$. Similarly, again by \rife{lip} and parabolic regularity, we have $\|Du(t)\|_\infty\leq C$, $t\in (0,T-1]$,  for some $C$ only depending on $L$ (the Lipschitz bound of $H$) and the global bounds of $m$, $\tilde u$. Therefore, $(u,m)$ satisfies \rife{MUT}  in the interval $(1,T-1)$, for some $\mathcal M$, $\mathcal U$ only depending on $L$. By Theorem \ref{uniq}, there exists $\gamma_L$ such that problem \rife{mfgT} admits at most one solution  in $(1,T-1)$ if $\gamma\leq \gamma_L$. This means that $(u,m)$ coincides with the solution built in Step 1 in the interval $(1,T-1)$, with initial-terminal conditions given by $m(1)$ and $u(T-1)$. Hence $(u,m)$ satisfies the exponential estimate \rife{turnp2}.
\qed

\vskip1em

We now observe that, as a consequence of the previous result, any globally bounded solution has a stationary attractor, provided the anti-monotonicity constant is sufficiently small.

\begin{cor}\label{boundT} Assume that $(u^T,m^T)$ is a solution of \rife{mfgT} which satisfies
$$
\sup_{[0, T] \times \T} |D u^T| \le \mathcal U
$$
for some $\U$ independent of $T$.  Assume that  $F,G,H$ satisfy conditions  \rife{fK}, \rife{gK} and \rife{liploc}-\rife{Hk} respectively.

Then, there exists some $\gamma$, only depending on $\kappa, \mathcal M, \mathcal U$, and on $F,H$ (through the local bounds  induced by $\mathcal M, \mathcal U$) such that 
if $s\mapsto F(x,s)+ \gamma s$ is nondecreasing we have
\be\label{newturn}
 \| m^T(t)-\bar m\|_\infty + \| Du^T(t)-D\bar u\|_\infty \leq M (e^{-\omega t}+ e^{-\omega (T-t)}) \qquad \forall t\in (1,T-1)\,,
\ee
for some $\omega, M>0$ (independent of $T$) and some   $(\bar u, \bar m)$  solution of the ergodic problem \rife{MFGergo}.
\end{cor}

\begin{proof}  We first build a globally Lipschitz  extension of  the Hamiltonian function $H(x,p)$. Namely, we consider a function $\tilde H(x,p)$ which still satisfies \rife{Hk} and such that
$$
\tilde H(x,p) \equiv H(x,p) \quad \hbox{for $|p| \leq \U$}\,,\qquad |H_p(x,p)| \leq C_\U\quad \forall (x,p) \in \R\times \R^d
$$
for some number $C_\U$ only depending on $\U$ (eventually through some constants  in \rife{liploc}-\rife{Hk} depending on $\U$). 

An example of a similar extension can be built as follows. First of all we take a cut-off function $\zeta\in C^2_c(\R^d)$ such that $\zeta \equiv 1 $ for $|p|\leq 2$ and $\zeta\equiv 0$ for $|p|>3$; then we take a $C^2$ real function $\vfi(r)$ such that $\vfi(r)\equiv 0$ if $|r|\leq \frac32$, $\vfi$ is increasing, locally uniformly convex for $r\in (\frac32, +\infty)$ and globally Lipschitz continuous. Then the function 
$$
\tilde H(x,p):= H(x,p) \zeta\left(\frac p{\U}\right) + C \vfi\left(\frac {|p|}\U\right)
$$
satisfies the required properties for  a convenient choice of $C$ (which will only depend on $\U$).

Now, we can apply Theorem \ref{longtime} to the MFG system
\begin{equation}\label{MFGtildeH}
\begin{cases}
-u_t - \kappa\Delta u + \tilde H(x, Du) = F(x, m(t)), & t 	\in (0, T) \\
m_t - \kappa \Delta m - {\rm div}(m \tilde H_p(x, Du)) = 0, & t \in (0, T) \\
m(x, 0)=m_0(x), \qquad u(x, T) = G(x, m(T))
\end{cases}
\end{equation}
so there exists some $\gamma>0$ such that if $s\mapsto F(x,s)+ \gamma s$ is nondecreasing, then the ergodic stationary problem (corresponding to $\tilde H(x,p)$) has a unique solution $(\bar m,\bar u)$
and any solution of \rife{MFGtildeH} satisfies \rife{turnp2}. This value of $\gamma$ depends on $\tilde H$, so it actually depends on $\U$ and on the functions $F,H$. Now, since $\|Du^T\|_\infty\leq \U$, we have $\tilde H(x,Du^T)=H(x,Du^T)$ so the result applies to the given solution of the original problem \rife{mfgT}. Therefore, $(u^T, m^T)$ satisfies \rife{newturn}. It remains only to show that $(\bar m, \bar u)$ is solution with $H(x,p)$ rather than with $\tilde H(x,p)$.  But this must necessarily be true, because \rife{newturn} implies, e.g., that $ Du^T\left(x,\frac T2\right) \to D\bar u(x)$, so $\|D\bar u\|_\infty \leq \U$. This also proves that $(\bar m, \bar u)$ does not depend on the extension
$\tilde H(x,p)$ which was constructed.
\end{proof}

\begin{rem} The previous result is new even in case that $F(x,\cdot)$ is nondecreasing. It means that any solution $(u^T,m^T)$ such that $Du^T$ is uniformly  bounded
exhibits a  turnpike  behavior in long horizon and approximates (for most of time) a  stationary solution $(\bar m, \bar u)$.  So far this shows that a global (in time) gradient bound for $u^T$ is  sufficient for the turnpike property to hold.
\end{rem}

\vskip0.5em
Corollary \ref{boundT} may be applied, for instance, to the case of superlinear Hamiltonian with quadratic growth when the cost function $F$ has a moderate growth in the $m$ variable, i.e. the model case (B) addressed in the previous Section.  The corollary below is in particular a consequence of Corollary \ref{boundT} and Theorem \ref{thmquadH}.

\begin{cor} Assume that $H$ and $F$ satisfy the conditions   \eqref{globHass} and \eqref{fK},  \eqref{fass}, respectively. For $m_0\in \elle\infty$ and $G$ satisfying \rife{gut2}, let $(u^T,m^T)$ be solution of \rife{mfgT}.

Then, there exists some $\gamma$, only depending on $\kappa$, $\|m_0\|_\infty$, and on $F,H$ (through constants appearing in the assumptions above, and the local bounds induced by $\mathcal M, \mathcal U$) such that 
if $s\mapsto F(x,s)+ \gamma s$ is nondecreasing we have
$$
 \| m^T(t)-\bar m\|_\infty + \| Du^T(t)-D\bar u\|_\infty \leq M (e^{-\omega t}+ e^{-\omega (T-t)}) \qquad \forall t\in (1,T-1)\,,
$$
for some $\omega, M>0$ (independent of $T$) and some   $(\bar u, \bar m)$  solution of the ergodic problem \rife{MFGergo}.
\end{cor}

\section{The stationary feedback and the convergence of $u^T$}

In this section we are going to see how the   exponential turnpike estimate established so far implies the convergence of $u^T(t)-\bar \lambda (T-t)$, as well as of $m^T(t)$, for any fixed $t>0$.  

From now on, and throughout the rest of the paper, we will consider only the case (A) discussed above, which involves globally Lipschitz and locally uniformly convex Hamiltonians. Similar results can be proven for case (B) involving uniformly convex Hamiltonians and costs with mild growth, in view of the estimates obtained in Section 3, but they will not be stated explicitly for brevity.

We are now interested in a convergence for any $t>0$, so we are going to require that \rife{turnp2} holds in the whole interval $[0,T]$ (this is true in Theorem \ref{longtime} if  $m_0\in \elle\infty$  and $u_T\in W^{1,\infty}(\T)$):
\be\label{stima-exp}
  \|m^T (t) - \bar m\|_{L^\infty} + \| D u^T (t) - D\bar u\|_{L^\infty} \leq 
K(e^{-\omega t} + e^{-\omega (T-t)} ) \qquad \forall  t \in [0, T]\,.
\ee
We start by deducing  the following global bound as a corollary of \rife{stima-exp}.

\begin{cor}\label{stima-inf} Let $H,F,G$ satisfy \rife{liploc}, \rife{fK}, \rife{gK} respectively.   Assume that $(u^T,m^T)$ is a solution of \rife{mfgT} which satisfies \rife{stima-exp}. Then there exists  a constant $K$, independent of $T$, such that
\be\label{infbound}
\|u^T-\bar \lambda (T-t)\|_{\infty} \leq K\qquad \forall t\in [0,T]\,.
\ee
\end{cor}

\begin{proof}  Due to estimate \rife{stima-exp}, and the local Lipschitz character of $F$, we have that
$$
|F(x,m^T)-F(x,\bar m)| \leq c\, (e^{-\omega t} + e^{-\omega (T-t)} )\qquad \forall t\in [0,T]\,.
$$
Moreover, using \rife{gK}, we  have that $G(x,m^T(T))$ is uniformly bounded. 
Hence, if we denote by $\vfi(t):= M(1+ \int_t^T (e^{-\omega s} + e^{-\omega (T-s)} )ds)$, we deduce that, for sufficiently large $M$, 
the functions $\bar u + \bar \lambda(T-t)+ \vfi(t)$ and $\bar u + \bar \lambda(T-t)- \vfi(t)$ are, respectively, super and sub solution to the equation satisfied by $u^T$. The comparison principle readily yields
$$
\bar u -\vfi(t) \leq u^T-\bar \lambda(T-t) \leq  \bar u + \vfi(t) \qquad \forall t\in [0,T], x\in \T
$$
which implies \rife{infbound}.
\end{proof}

Now we establish a uniqueness result for the limiting problem of the MFG system as $T\to \infty$; namely, we consider the problem
\be\label{MFG-infty}
\begin{cases}
-v_t + \bar \lambda - \kappa\Delta v + H(x, D v)= F(x,\mu(t))\,, & \hbox{$t\in (0,\infty)$}
\\
\mu_t- \kappa \Delta \mu -\dive( \mu\, H_p(x,D v))= 0\,,& \hbox{$t\in (0,\infty)$}
\\
\mu(x,0)= m_0(x)\,, \\
v\in L^\infty((0,\infty)\times \T)\,,\,\, Dv\in D\bar u + L^2((0,\infty)\times \T)\,,\, 
\end{cases}
\ee

\begin{lem}\label{uni-lim} Let $m_0\in \elle\infty$, $m_0\geq 0$. Assume that $H$ satisfies \rife{Hk} and \rife{lip}, and $F$ satisfies \rife{fK}.  There exists $\gamma>0$, only depending on $\|m_0\|_\infty, \kappa, L$ (appearing in \rife{lip}) and on $H,F$ (through \rife{Hk}, \rife{fK} for some $K$ only depending on $\|m_0\|_\infty, \kappa, L$)
such that, if $F(x,s)+\gamma s$ is nondecreasing, then there are at most one $\mu$ and  one $v$ (up to addition of a constant) which solve problem \rife{MFG-infty}.
\end{lem}

\begin{proof} We first observe that $\into \mu(t)= \into m_0$ for every $t$, so $\|\mu(t)\|_{\elle1}$ is uniformly bounded;  since $H_p$ is uniformly bounded by $L$, this implies that $\mu \in L^\infty((0,\infty)\times \T)$ and its norm is universally bounded by a constant only depending on $d,\kappa, L, \|m_0\|_\infty$. This provides with a uniform bound for $F(x,\mu)$ and in turn, as a consequence of Lemma \ref{lem72}, this implies that $\tilde v(t)$ and $Dv(t)$ are bounded in $\elle2$, uniformly in time.  
By local parabolic regularity we  deduce that $Dv\in L^\infty((0,\infty)\times \T)$ as well and its norm is bounded uniformly in time, for all solutions. 
We will therefore use the  upper and lower bounds of $H_{pp}$ for $|p|$ varying in a compact set {\it which is the same for all solutions $(\mu, v)$}.  

On account of the above ingredients, we follow  a standard argument for the proof of uniqueness.
Let $(\mu_1 , v_1 ), (\mu_2 , v_2 )$ be two solutions of \rife{MFG-infty}.
Consider a function $\xi_R (t) := \xi(t/R)$, where $\xi$ is a $C^1$ function such that $\xi \equiv 1$ in $(0, 1)$ and
with support in $(-1, 2)$. Using $\psi = (\mu_1 - \mu_2 )\xi_R$ as test function in the equation of $v_1 - v_2$, we
get as usual:
\begin{align*}
-\frac{d}{dt}  & \left(\xi_R(t) \into  (\mu_1 - \mu_2 )(v_1-v_2)dx  \right) = - \xi_R'(t)\into  (\mu_1 - \mu_2 )(v_1-v_2)dx \\
& + \xi_R(t) \into [F(\mu_1)-F(\mu_2)][\mu_1-\mu_2] \, dx 
\\
& + \xi_R(t) \into \mu_1\left[H(x,Dv_2)- H(x, Dv_1)- H_p(x, Dv_1) D(v_2-v_1)\right]dx
\\
& + \xi_R(t) \into \mu_2\left[H(x,Dv_1)- H(x, Dv_2)- H_p(x, Dv_2) D(v_1-v_2)\right]dx
\end{align*}
which implies, using the uniform convexity of $H$:
\begin{align*}
-\frac{d}{dt} \left(\xi_R(t) \into  (\mu_1 - \mu_2 )(v_1-v_2)dx \right) & \geq  - \xi_R'(t)\into  (\mu_1 - \mu_2 )(v_1-v_2)dx \\
& + \xi_R(t) \into [F(\mu_1)-F(\mu_2)][\mu_1-\mu_2] \, dx 
\\
& + c\, \xi_R(t) \into (\mu_1+ \mu_2) |D(v_1-v_2)|^2 dx\,,
\end{align*}
for some constant $c>0$.  
Integrating we get
\begin{align*}
c\, \int_0^\infty\xi_R(t) \into \mu_2\, |D(v_1-v_2)|^2 dxdt \leq \int_0^\infty \xi_R'(t)\into  (\mu_1 - \mu_2 )(v_1-v_2)dxdt \\
  + \gamma\,  \int_0^\infty \xi_R(t) \into |\mu_1-\mu_2|^2 \, dx dt
\end{align*}
where we used that $F(x,s)+ \gamma  s$  is monotone. Notice that $Dv_1-Dv_2\in L^2((0,\infty)\times \T )$ by  assumption.  Then, by Lemma \ref{stimaRhoF} (see \rife{77} with $\de=0$), we deduce
\[
\int_{0}^\infty \|(\mu_1-\mu_2)(t)\|^2_2 dt \le C   \int_{0}^\infty \int \mu_2^2 \, |D(v_1-v_2)|^2 dxdt  \leq C\int_{0}^\infty \int \mu_2  \, |D(v_1-v_2)|^2 dxdt
\]
for some constant $C$  depending on $\|H_{p}\|_\infty$, on    the   $L^\infty$ bound of $\mu_2$ and on the local upper bound of $H_{pp}$ (all being estimated only in terms of $\kappa, L, \|m_0\|_\infty$). Hence we deduce
\begin{align*}
(c- \gamma\, C) \int_0^\infty \xi_R(t) \into \mu_2(t) |D(v_1-v_2)|^2 dxdt \leq \int_0^\infty \xi_R'(t)\into  (\mu_1 - \mu_2 )(v_1-v_2)dxdt
\\
+ c \int_0^\infty (1-\xi_R(t)) \into \mu_2(t)  |D(v_1-v_2)|^2 dxdt\,.
\end{align*}
Using Poincar\'e-Wirtinger inequality and  the fact that $\mu_1 -\mu_2 $ is uniformly bounded in $\elle2$, due to the properties of $\xi_R(t)$  we get
\begin{align*}
(c- \gamma\, C) \int_0^\infty\xi_R(t) \into \mu_2(t) |D(v_1-v_2)|^2 dxdt \leq c \, \left(\frac1R \int_R^{2R} \into   |D(v_1-v_2)|^2dxdt\right)^{\frac12}
\\
+ c \int_0^\infty (1-\xi_R(t)) \into \mu_2(t) |D(v_1-v_2)|^2 dxdt 
\end{align*}
and  letting $R\to \infty$ we conclude 
$$
\int_0^\infty \int \mu_2(t) |D(v_1-v_2)|^2 dxdt=0\, 
$$
for a sufficiently small $\gamma$, only depending on $\|m_0\|_\infty, \kappa, L$ (eventually trough $F,H$).

Hence $Dv_1=Dv_2$ and, from the Fokker-Planck equation, this yields $\mu_1=\mu_2$. Finally, this implies $(v_1-v_2)_t=0$, hence $v_1(t,x)=v_2(t,x)+K$ for some constant $K\in \R$, and for every $(t,x)$. 
\end{proof}

We now establish the convergence of $u^T-\bar \lambda (T-t)$.

\begin{thm}\label{convuT}  Let $m_0\in \elle\infty$. Assume that $F(x,m)$  satisfies 
\rife{fK}, $ H(x,p)$  satisfies \rife{Hk} and \rife{lip},  and the final cost  $G$ satisfies \rife{gut}. Let $(u^T,m^T)$ be a solution of problem \rife{mfgT}.  

There exists $\gamma>0$,  only depending on $\|m_0\|_\infty, \kappa, L$ (and on the functions $F,H$), such that  if  $F(x,s) + \gamma s$
is nondecreasing,  then 
there exists a solution $(v,\mu)$ of \rife{MFG-infty} such that
$$
\lim\limits_{T\to \infty} u^T(x,t)-\bar \lambda(T-t)   = v(x,t)\,\qquad ; \qquad \lim\limits_{T\to \infty} m^T(x,t) = \mu(x,t) \qquad \forall t>0, x\in \T
$$
and the convergence is uniform in $[a,b]\times \T$, for any $[a,b]\subset [0, \infty)$.
\vskip0.5em
In particular, there exists a constant $\bar c\in \R$ such that $(v,\mu)$ is the unique solution of \rife{MFG-infty} such that $\lim\limits_{t\to \infty} \int v(t) = \bar c$.
\end{thm}

\begin{proof} We set $v^T:= u^T(x,t)-\bar \lambda(T-t)$. Since $m_0\in \elle\infty$  and $u(T) \in W^{1,\infty}(\Omega)$, we know that the exponential estimate \rife{turnp2} holds for $t\in [0,T]$.  Hence, from Corollary \ref{stima-inf}, we have that $\|v^T\|_\infty$ and $\|Dv^T\|_\infty$ are  bounded 
uniformly with respect to $T$; and by parabolic regularity,  this readily implies that $v^T$ is locally (in time) relatively compact 
in the uniform topology. Similarly, $m^T$ is locally bounded in H\"older (time-space) norms, and it is relatively compact in the uniform topology (locally in time). Then, it is straightforward to see that $(v^T, m^T)$ converges, up to subsequences, to a solution $(v,\mu)$ of \rife{MFG-infty}. We only need to prove that the limit function $v$ is the same for all subsequences.

To this purpose, we develop an idea  from \cite{CaPo}:  for $T,T'$ we consider $v^T, v^{T'}$ and we define the shifted functions
$$
(\hat v^T, \hat \mu^T):= (v^T(t+T), m^T(t+T))\,,\,\,\, \quad (\hat v^{T'}, \hat \mu^{T'}):= (v^{T'}(t+T'), m^{T'}(t+T'))\,.
$$
Both $(\hat v^T, \hat \mu^T)$ and $(\hat v^{T'}, \hat \mu^{T'})$ are solutions of MFG systems in the interval $(-\tau,0)$; the usual 
estimate gives
$$
-\frac{d}{dt}      \into  (\hat \mu^T - \hat \mu^{T'} )(\hat v^T-\hat v^{T'})dx   \geq   
   -\gamma \into (\hat \mu^T - \hat \mu^{T'})^2 \, dx 
   + c\,   \into (\hat \mu^T + \hat \mu^{T'})|D(\hat v^T-\hat v^{T'})|^2 dx\,,
$$
where we used that $F(x,s)+ \gamma s$ is monotone and the uniform bound from below of $H_{pp}$ on compact subsets.
Integrating and using the final condition, 
we have
\begin{align*}& 
c \int_{-\tau}^0\into (\hat \mu^T + \hat \mu^{T'})|D(\hat v^T-\hat v^{T'})|^2 dx \leq   \gamma \int_{-\tau}^0 \into (\hat \mu^T - \hat \mu^{T'})^2 \, dx 
\\
& \qquad +  \into  (\hat \mu^T(-\tau) - \hat \mu^{T'}(-\tau) )(\hat v^T(-\tau)-\hat v^{T'}(-\tau))dx
\\
& \qquad \leq   \gamma \int_{-\tau}^0 \into (\hat \mu^T - \hat \mu^{T'})^2 \, dx 
+\|m^T (T-\tau) - m^{T'}(T'-\tau)\|_{L^2}\,  \| Du^T (T-\tau) - Du^{T'} (T'-\tau)\|_{L^2}
\end{align*}
which yields, thanks to \rife{stima-exp},
\begin{align*}
&   \int_{-\tau}^0\into (\hat \mu^T + \hat \mu^{T'})|D(\hat v^T-\hat v^{T'})|^2 dx  \leq c\, \gamma  \int_{-\tau}^0 \into (\hat \mu^T - \hat \mu^{T'})^2 \, dx
\\
& \qquad \qquad + c (e^{-\omega \tau} + e^{-\omega (T-\tau)} + e^{-\omega (T'-\tau)})^2
\end{align*}
where, as usual,  we denote generically by $c$ possibly different constants (which may vary from line to line)  independent of $T, T'$.
Applying Lemma \ref{stimaRhoF}  to the equation of $\hat \mu^T - \hat \mu^{T'}$,  we have
\begin{align*}
\int_{-\tau}^0 \int (\hat \mu^T - \hat \mu^{T'})^2 \, dx  & 
\leq  C\, \int_{-\tau}^0\int (\hat \mu^{T})^2|D(\hat v^T-\hat v^{T'})|^2 dx  
\\  & \qquad + C \|m^T (T-\tau) - m^{T'}(T'-\tau)\|_{L^2}^2\,.
\end{align*}
Thus for $\gamma$ sufficiently small we conclude that
$$
\int_{-\tau}^0\int (\hat \mu^T + \hat \mu^{T'}) |D(\hat v^T-\hat v^{T'})|^2 dx  \leq  c \, (e^{-\omega \tau} + e^{-\omega (T-\tau)} + e^{-\omega (T'-\tau)})^2\,.
$$
A similar estimate is then deduced for   $\sup_{t\in [-\tau,0]} \|\hat \mu^T(t) - \hat \mu^{T'}(t)\|_{2}$, from the above estimates.

With a  bootstrap argument, using the global $L^\infty$ bounds for $D(v^T-v^{T'})$ and for $m^T  - m^{T'}$, the previous $L^2$ bound can be updated into  a $L^\infty$ bound:
$$
\|\hat \mu^T(t) - \hat \mu^{T'}(t)\|_\infty \leq c \,(e^{-\omega \tau} + e^{-\omega (T-\tau)} + e^{-\omega (T'-\tau)}) \qquad \forall t\in [-\tau,0]\,.
$$
In particular, those bounds are inherited by $F(x,\hat \mu^T)-F(x,\hat \mu^{T'})$ as well as by $G(x,\hat \mu^T(0))-G(x,\hat \mu^{T'}(0))$. 
Therefore, the maximum principle implies that 
$$
\|\hat v^T(t)-\hat v^{T'}(t)\|_\infty \leq c\, \tau (e^{-\omega \tau} + e^{-\omega (T-\tau)} + e^{-\omega (T'-\tau)})\qquad \forall t\in [-\tau,0]\,.
$$
In particular, we have proved that
\be\label{cauchy}
\|v^T(T-\tau)-  v^{T'}(T'-\tau)\|_\infty \leq c \tau (e^{-\omega \tau} + e^{-\omega (T-\tau)} + e^{-\omega (T'-\tau)})\,.
\ee
Now we consider the equation of $v^T-\bar u$ in the interval $(t,T-\tau)$; integrating we have
$$
\into v^T(t)- \into v^T(T-\tau) + \int_t^{T-\tau} \into [H(Du^T)- H(D\bar u)]dxds = \int_t^{T-\tau} \into [F(m^T)- F(\bar m)]dxds
$$
which yields, using \rife{fK}, \rife{lip} and \rife{stima-exp},  
$$
\left|\into v^T(t)- \into v^T(T-\tau)\right| 
\leq c \int_t^{T-\tau} (e^{-\omega s} + e^{-\omega (T-s)} ) ds \leq c \, (e^{-\omega t} + e^{-\omega \tau} )\,.
$$
Similarly we integrate the equation of $v^{T'}-\bar u$ in the interval $(t,T'-\tau)$ and we get
$$
\left|\into v^{T'}(t)- \into v^{T'}(T'-\tau)\right| 
\leq  c  (e^{-\omega t} + e^{-\omega \tau} )\,.
$$
Putting the latter two inequalities together with \rife{cauchy}, we conclude that
\be\label{medie}
\begin{split}
\left|\into v^T(t)- \into v^{T'}(t)\right| &  \leq c \, (e^{-\omega t} + e^{-\omega \tau} )
+ c\,  \tau (e^{-\omega \tau} + e^{-\omega (T-\tau)} + e^{-\omega (T'-\tau)})\,.
\end{split}
\ee
Finally, we consider two possible limits $v_1, v_2$ obtained with different subsequences $v^{T_n}$ and $v^{T'_n}$; since both 
are solutions of \rife{MFG-infty}, we have that $v_1-v_2$ is a constant. However, passing to the limit as $T_n,T_n'\to \infty$ in \rife{medie}, and then letting $\tau \to \infty$, we obtain
$$
\left|\into v_1(t)- \into v_2(t)\right|    \leq c  \,e^{-\omega t} \,.
$$
Letting $t\to \infty$ implies that  $v_1-v_2$ can only be the null constant. This proves that the limit of $v^T$ is independent of the subsequence, hence the whole sequence $v^T$ converges. The convergence of $m^T$ follows itself from the uniqueness result of Lemma \ref{uni-lim}.

Finally, by integrating the equation of $v-\bar u$, we have
$$
\into v(t)- \into v(t') + \int_t^{t'} \into [H(Dv)- H(D\bar u)]dxds = \int_t^{t'} \into [F(\mu)- F(\bar m)]dxds\,.
$$
Since $v,\mu$ satisfy the estimate (a consequence of \rife{stima-exp})
$$
\|Dv(t)-D\bar u\|_\infty+ \|\mu(t)-\bar m\|_\infty \leq K \,e^{-\omega t}\,,
$$
we deduce that 
$$
\left |\into v(t)- \into v(t') \right| \leq   K \int_t^{t'} e^{-\omega s}ds \quad \mathop{\to}^{t,t'\to \infty} \,\,\,0
$$
hence $\into v(t)$ is a Cauchy sequence and admits a limit as $t\to \infty$. The value of this limit, say $\bar c$, fully characterizes the function $v$ due to Lemma \ref{uni-lim}.
\end{proof}

\vskip1em
\begin{rem}\label{Grema}
In the above result, we have assumed that the final cost $G$ is independent of $m$. The reason is that if $G$ is just a Lipschitz function of the density $m(T)$, then we are not able to show a  bound for $Du^T$ up to $t=T$. Otherwise, should we have a global bound for $Du$ in the whole $(0,T)$, then  the same conclusion would hold for $G$ satisfying \rife{gK} and $G(x,m)+ \gamma m$ monotone with $\gamma $ small. 

We stress, in particular, that the convergence result of Theorem \ref{convuT} remains true for smoothing couplings at final time, say for instance if  $G=G(x,m)$ is a Lipschitz continuous mapping from $\T\times \cP(\T)$ (endowed with the Wasserstein distance $d_1$) such that $\|G(x,m)\|_{W^{1,\infty}(\T)}$ is bounded uniformly in $\cP(\T)$ and 
$$
\into (G(x,m_1)- G(x,m_2)d(m_1-m_2) \geq - \gamma \,\|m_1- m_2\|_{\elle1}^2 \qquad \forall m_1, m_2 \in \cP(\T) \cap \elle1
$$
for $\gamma$ sufficiently small.
\end{rem}

\vskip1em

\begin{rem} It would be possible to characterize the limit of $u^T-\bar \lambda(T-t)$ in terms of a stationary feedback operator defined on the current measure $\mu(t)$, which is the unique solution of \rife{MFG-infty} (see Lemma \ref{uni-lim}). Namely, one can define an operator $\hat E$ such that
$$
\lim\limits_{T\to \infty} u^T(x,t)-\bar \lambda(T-t)= \hat E(\mu(t))\qquad \hbox{for any $t>0$.}
$$
In the present setting, $\hat E$ could only be defined as an unbounded operator in $(\elle2)_+$, with a  domain which includes the set of bounded functions. Nevertheless, $\hat E$ could still be characterized thanks to the well-posedness of the MFG system. We stress that  this kind of feedback (which plays a similar role as the Riccati stationary operator in other control problems) would coincide with a solution of the stationary ergodic  master equation introduced in \cite{CaPo} for monotone and smooth mean field game systems. This may suggest alternative ways to look at the master equation of mean field games whenever it cannot be defined as a  smooth function on the space of probability measures.
\end{rem}

\section{The discounted problem}

We now deduce the existence of a   solution to the infinite horizon MFG system. 
As before, we start by  assuming that $p\mapsto H(x,p)$ is a $C^2$  function which satisfies   \rife{Hk} and \rife{lip}, and that   $F(x,m)$  satisfies 
\rife{fK} and
\be\label{f3}
 (F(x,m)-F(x,m'))(m-m')\geq - \gamma (m-m')^2 \qquad \forall x \in \T \,, m,m' \in \R\, 
\ee
for some $\gamma>0$.

It is possible to prove that, under the above assumptions,  there exist  $\de_0, \gamma_0>0$ such that if $\gamma<\gamma_0$ and $\de<\de_0$ then the stationary problem 
\be\label{ergde}
\begin{cases}
  \de u -  \kappa  \Delta u + H(x,Du)= F(x,m) &  x\in \T
\\
  -  \kappa \Delta m- \dive(m\, H_p(x,Du)) = 0 &  x\in \T
\\
\into m = 1\,  
& \end{cases}
\ee
admits a unique solution $(\bar u_\de, \bar m_\de)$, which is smooth. Here $\gamma_0$ only depends on $\kappa, H,F$, i.e. it depends on the constant $L$ in \rife{lip} and on the constants $c_K$, $\ell_K$, $\alpha_K, \beta_K$ for a $K$ only depending on $\kappa, L$.  While the existence of $(\bar u_\de, \bar m_\de)$ can be proved with a usual fixed point method, the uniqueness  argument is  similar  as  the one we   used in Theorem \ref{uniq} and requires smallness of $\de$ as well.

\begin{thm}\label{discex} Let $m_0\in \elle\infty$. Assume that  $H$ satisfies  \rife{Hk} and \rife{lip}, $F$ satisfies \rife{fK} and \rife{f3}.  There exist $ \gamma_0, \delta_0>0$ such that  if \rife{f3} holds with $\gamma<\gamma_0$, and if $\de<\de_0$, then  there exists a unique solution to the infinite horizon problem
\be\label{ih}
\begin{cases}
-u_t + \de u -   \kappa \Delta u + H(x,Du)= F(x,m) &  t\in (0,\infty)
\\
m_t -  \kappa  \Delta m- \dive(m\, H_p(x,Du)) = 0 & t\in (0,\infty)
\\
m(0)= m_0\,,\qquad u\in L^\infty((0,\infty)\times \T)\,. 
& \end{cases}
\ee
In addition,   there exist  $ M,\omega>0$ such that 
\be\label{expde}
\| m(t)-  \bar   m_\de\|_\infty+ \| Du(t)-D  \bar u_\de\|_\infty \leq M\, e^{-\omega t} \|m_0-  \bar m_\de\|_\infty \qquad \forall t>0\,,\,\,\forall \de<\de_0. 
\ee
The constants $\gamma_0, \de_0$, as well as $M,\omega$, only depend  on $\kappa, \|m_0\|_\infty$, the constant $L$ in \rife{lip}  and the local growth constants of $F,H$ given by \rife{Hk}, \rife{fK} (for some $K$ depending on $\kappa, \|m_0\|_\infty, L$).
\end{thm}

\vskip1em

\begin{rem}   We stress that the smallness condition required on $\gamma$ in Theorem \ref{discex} does not depend on the parameter $\de$. Indeed, the purpose of this result is to provide an exponential decay in time which is uniform for $\de$ sufficiently small, in order to apply it to the vanishing discount limit studied in the next Section. 
Such a uniform decay rate is obtained in the next Lemma, where we choose $\de\leq \de_0$ in the final step; no effort is made here to quantify $\de_0$ (in fact a  possible study for arbitrary large $\de$ would even be possible but is beyond our goals here).
\end{rem}

\vskip1em
The proof of Theorem \ref{discex} will follow from a  fixed point argument similarly as in Theorem \ref{longtime}.  The crucial step consists in obtaining  a priori estimates on the system:
 \be\label{lin-dis}
 \begin{cases}
-v_t + \de v- \kappa \Delta v + h(x, Dv)  = f(t,x,  \mu)   & 
\\
\mu_t - \kappa \Delta \mu- \dive(\mu\, h_p(x, Dv)) = \sigma \, \dive(B(x,Dv)) & 
\\
\mu(0)= \sigma \mu_0\,,\,\,\,\, \mu,v\in L^\infty((0,\infty)\times \T)\,
& \end{cases}
\ee
where $\sigma\in [0,1]$, $\mu_0\in \elle\infty$,  with $\into \mu_0=0$, and where the functions $h(x,p)$, $f(t,x, s)$ and $B(x,p)$ satisfy the conditions \rife{hT}-\rife{BT}, which we rewrite below for the reader's convenience.

\begin{lem}\label{apriori-dis}
Let $h(x,p)$, $f(t,x, s)$ and $B(x,p)$ be   continuous functions satisfying  the conditions
 \be\label{h}
 h(x,0)=0\,,\quad |h_p(x,p)| \leq \ell_0\,,
 \ee
 \be\label{f}
 f(t,x,s)s\geq -\gamma \,s^2\,,\quad |f(t,x,s)|\leq C_1 \, |s|
 \ee
\be\label{B}
B(x,p)\cdot p \geq C_2^{-1} |p|^2\,, \qquad |B(x,p)|\leq C_2\, |p| \,.
 \ee
  for every 
$s\in \R$, $t>0$, $x\in \T$, $p\in \R^d$ such that $|p| \leq K$.  
For $\sigma\in [0,1]$, $\mu_0\in \elle\infty$,  with $\into \mu_0=0$, let $(\mu,v)$ be a solution of \rife{lin-dis} which satisfies, for any $(t,x)\in Q_T$,  that $|Dv(t,x)|\leq K$ and 
\be\label{summa2}
\begin{split}
\sigma B(x,p)\cdot p -   \mu(t,x) (h(x,p)-h_p(x,p)\cdot p)      \geq \sigma\,c_0\,  |p|^2\qquad \forall (t,x)\in Q_T, \forall p\,:|p|\leq K\,, 
\end{split}
\ee
for some $c_0>0$. 
 
 Then there exist constants $\gamma_0, \de_0, \omega, M>0$ (only depending on  $\kappa, \ell_0, C_1, C_2, c_0$)  such that, if $\gamma\leq \gamma_0$ ($\gamma$ is in \rife{f}),  if $\de\leq \de_0$ and $\sigma\in [0,1]$,   then  $(\mu,v)$   satisfies
 $$
 \| \mu(t)\|_{2}+ \| Dv(t)\|_2 \leq M  e^{-\omega t} \,\|\mu_0\|_2  \qquad \forall t>0\,.
 $$
\end{lem}

\proof  First  we observe that conditions \rife{f}, \rife{B} imply, respectively, $f(x,0)=0$ and $B(x,0)=0$. The proof of the estimate is divided in two steps.
\vskip0.5em
{\it Step 1} In this first step, we show that there exists $c$   such that 
\be\label{step1}
\| \mu(t)\|_{2}+ \| Dv(t)\|_2 \leq c\, \|\mu_0\|_2 \, e^{\frac \de2 t}\,,
\ee
and
\be\label{muinf}
\int_0^\infty \into   e^{-\de s}\mu^2(s)\,ds  + \sigma \int_0^{\infty}  \into e^{-\de s}    |Dv|^2 \leq c\,   \|\mu_0\|_2^2\,.
\ee
To show the above two estimates, we observe that the duality between the two equations implies
\be\label{mon}
\begin{split}-\frac{d}{dt} \, \left[e^{-\de t}\, \into \mu\, v\right] & =  e^{-\de t}\left\{\into f(x,\mu)\mu +\sigma  \into B(x,Dv)Dv+ \into \mu(h_p(x,Dv)\cdot Dv-H(xDv))\right\} 
\\
& \geq -\gamma e^{-\de t} \into |\mu(t)|^2 +\sigma \, c_0 \, e^{-\de t} \into |Dv|^2
\end{split}
\ee
where we used \rife{f} and \rife{summa2}.  Hence  we get
$$
c_0\sigma \int_0^t  \into e^{-\de s}    |Dv|^2\leq  \gamma \int_0^t  \into e^{-\de s} \mu^2 + \into \mu_0 v(0) - e^{-\de t} \into  \mu(t)v(t)\,.
$$
Now we use \rife{77} from  Lemma \ref{stimaRhoF}; since $|B(x,Dv)|\leq C_2 \,|Dv|$  we deduce
\be\label{mud0}
\int_{t_0}^t  \into   e^{-\de s}\mu^2  \leq C\, e^{-\de t_0}\,  \|\mu(t_0)\|_2^2 + C\, \sigma^2\, C_2^2  \int_{t_0}^t \into e^{-\de s}\, |Dv|^2 \,,\quad \qquad \forall t>t_0\geq 0
\ee
for some constant $C$ independent of $\de$. Hence we get, using $\sigma\leq 1$:
$$
(c_0 - \gamma\, C\, C_2^2)\sigma \int_0^t  \into e^{-\de s}    |Dv|^2\leq  C\,\gamma   \|\mu_0\|_2^2+  \into \mu_0 v(0) - e^{-\de t} \into  \mu(t)v(t)\,.
$$
Since $\mu,v$ are globally bounded, we have that last term vanishes as $t\to \infty$. We deduce that, for $\gamma$ so that $c_0 - \gamma\, C\, C_2^2>0$, we have 
\be\label{sdv}
\sigma\int_0^{\infty}  \into e^{-\de s}    |Dv|^2\leq  c\,  [  \|\mu_0\|_2^2+ \| \mu_0 \|_2 \|\tilde v(0)\|_2] 
\ee
where, we recall, $\tilde v= v-\lg v\rg$.  Here and after we denote by  $c$ possibly different constants which only depend on $\kappa, \ell_0, C_1, C_2, c_0$.  

By using Lemma \ref{lem72} for the equation of $v$ (with horizon $T\to\infty$) we have
\be\label{tilv}
\begin{split}
\|\tilde v(t)\|_2 e^{-\de t}  & \leq C \,  \int_t^\infty e^{-\nu (s-t)}e^{-\de s}\|f(s, \cdot, \mu(s))\|_2ds
\\
& \leq C\, C_1   \int_t^\infty e^{-\nu (s-t)}e^{-\de s} \|\mu(s)\|_2ds
\end{split}
\ee
where $C$ only depends on $\kappa, \ell_0$ and we also used \rife{f}. If we take here $t=0$ and we estimate the right-hand side with   \rife{mud0}, then we obtain an estimate for $\|\tilde v(0)\|_2$, which can be used in  \rife{sdv} to deduce that 
$$
\sigma \int_0^{\infty}  \into e^{-\de s}    |Dv|^2\leq  c\,   \|\mu_0\|_2^2\,.
$$
In turn again from \rife{mud0} this concludes the proof of \rife{muinf}. 
Now, since  we have, from Lemma \ref{stimaRhoF},
\be\label{mud}
\|\mu(t)\|_2^2 \leq C\,  e^{-\nu t} \|\mu_0\|_2^2 + C\, \sigma^2 \, C_2^2 \, e^{\de t} \int_0^t \into e^{-\de s}\, |Dv|^2
\ee
we also deduce from \rife{muinf} that
$$
\|\mu(t)\|_2 \leq c \|\mu_0\|_2 \,e^{\frac \de2 t} \qquad \forall t>0 
$$
and in turn now \rife{tilv} implies as well:
$$
\|\tilde v(t)\|_2\leq c \|\mu_0\|_2 \,e^{\frac \de2 t} \qquad \forall t>0 \,.
$$
From Lemma \ref{lem72}, the above two estimates imply a similar one for $\|Dv(t)\|_2$, so \rife{step1} is proved.
\vskip0.4em
{\it Step 2.}  We first deduce from \rife{muinf} that there exist $\xi_\tau\in (0,\tau), \eta_\tau\in (2\tau,3\tau)$ such that
\be\label{muxitaubis}
e^{-\de \xi_\tau}\|\mu(\xi_\tau)\|_2^2 \leq  \frac c{\tau} \|\mu_0\|_2^2\,,\quad e^{-\de \eta_\tau}\|\mu(\eta_\tau)\|_2^2 \leq  \frac c{\tau} \|\mu_0\|_2^2\,.
\ee
Observe that, from \rife{mon}, we have
\begin{align*}
\sigma\, c_0 \int_{\xi_\tau}^{\eta_\tau} e^{-\de s}\into |Dv|^2&  \leq  \gamma \int_{\xi_\tau}^{\eta_\tau} e^{-\de s} \into |\mu|^2 
+ e^{-\de \xi_\tau} \into \mu(\xi_\tau)v(\xi_\tau) - e^{-\de \eta_\tau} \into \mu(\eta_\tau) v(\eta_\tau)
\\
& \leq \gamma \int_{\xi_\tau}^{\eta_\tau} e^{-\de s} \into |\mu|^2 + \frac c{\sqrt \tau}\, \|\mu_0\|_2^2
\end{align*}
where we used the global bound for $\tilde v(t)$ and \rife{muxitaubis}. Using now \rife{mud0} to estimate last integral in the right-side we get
$$
\sigma\, c_0 \int_{\xi_\tau}^{\eta_\tau} e^{-\de s}\into |Dv|^2  \leq C\,\gamma \,  e^{-\de \xi_\tau} \|\mu(\xi_\tau)\|_2^2 + \gamma C\, \sigma^2\, C_2^2  \int_{\xi_\tau}^{\eta_\tau} e^{-\de s}\into |Dv|^2 + \frac c{\sqrt \tau}\, \|\mu_0\|_2^2\,.
$$
We use \rife{muxitaubis} and we take  $\gamma$ sufficiently small (independent of $\de$) so we conclude that 
$$
  \int_{\xi_\tau}^{\eta_\tau} e^{-\de s}\into |Dv|^2  \leq  \left(\frac c\tau+\frac c{\sqrt \tau}\right)\, \|\mu_0\|_2^2\,.
$$
Therefore, for every $t\in [\tau, 2\tau]$ we estimate $\mu$ as
\begin{align*}
\|\mu(t)\|_2^2 & \leq  C e^{-\nu(t-\xi_\tau)}\|\mu(\xi_\tau)\|_2^2 + C \, C_2^2 \sigma^2 \int_{\xi_\tau}^{\eta_\tau} \into |Dv|^2
\\
& \leq C e^{-\nu(t-\xi_\tau)}e^{\de \xi_\tau} \frac{ c\|\mu_0\|_2^2}{\tau}+ C \, C_2^2 \sigma^2 e^{\de \eta_\tau} \int_{\xi_\tau}^{\eta_\tau}e^{-\de s} \into |Dv|^2  \,.
\end{align*}
Using  the estimate for the last integral and the fact that  $\eta_\tau \leq 3\tau$, we conclude that
$$
\|\mu(t)\|_2^2 \leq   e^{3\de\tau}  \left(\frac c\tau+\frac c{\sqrt \tau}\right)\, \|\mu_0\|_2^2\,.
$$
In particular, there exist $\tau_0$ (and correspondingly, $\de_0$) such that
$$
\|\mu(t)\|_2  \leq \frac 12 \|\mu_0\|_2  \qquad \forall t\in [ \tau_0, 2\tau_0], \quad \forall \de\leq \de_0\,.
$$
Iterating this estimate we conclude that there exists $\omega>0$ such that 
$$
\|\mu(t)\|_2  \leq e^{-\omega t} \|\mu_0\|_2\,,\quad \forall t>0\,,\quad \forall \de\leq \de_0\,.
$$
%
Finally, using \rife{tilv} we deduce a similar exponential decay for $\|\tilde v(t)\|_2$, and then for $\|Dv(t)\|_2$ as well.
\qed

\vskip1em

 We are now ready to prove Theorem \ref{discex}
 \vskip1em
 
{\bf Proof of Theorem \ref{discex}.}   Let us set $X= L^\infty((0,\infty); L^2_0(\T))$ where, we recall, $L^2_0(\T)$ denotes $L^2$ functions with zero average. 

We define the operator $\Phi_\vep$ on $X$ as follows: given $\mu\in X$, let $(v,\rho)$ be the solution to the system
\be\label{sysl}
\begin{cases}
-v_t +\de v - \kappa\Delta v + H(x,D  \bar u_\de + Dv) - H(x,D  \bar u_\de)= [F(x,  \bar m_\de + \mu) - F(x,  \bar m_\de)]e^{-\vep t} & 
\\
\rho_t - \kappa\Delta \rho- \dive(\rho\, H_p(x,D  \bar u_\de + Dv)) =  \dive(  \bar m_\de \left[ H_p(x,D \bar u_\de + Dv)- H_p(x, D  \bar u_\de)\right]) & 
\\
\rho(0)= m_0- \bar m_\de
& \end{cases}
\ee 
then we set $\rho= \Phi_\vep \mu$.  Here  $\vep>0$ is a parameter used in a  first step for compactness issues. 
We notice that  $\mu$ is a  fixed point of $\Phi_0$ if and only if $(v+\bar u_\de, \mu+\bar m_\de)$ solves the MFG system \rife{ih}.

In fact, if $m:=\rho+ \bar m_\de$,    then $m$ solves the evolution equation
$$
m_t- \Delta m - \dive(m H_p(x, D\bar u_\de+ Dv))=0\,.
$$
Since $H_p$ is bounded by $L$, and since $m_0\in \elle\infty$, $\into m(t)=1$ for every $t$, there exists a constant $M_0$ such that $\|m(t)\|_\infty\leq M_0$ for every $t>0$. The same applies to $\bar m_\de$. We deduce that the range of $\Phi_\vep$ is contained in a uniform ball in $L^\infty((0,\infty)\times \T)$. Hence, there is no loss of generality in restricting the domain of $\Phi_\vep$ to functions which are uniformly bounded (it would be enough to replace $\mu$ with a suitable truncation in the definition of the operator); in particular, due to \rife{fK}, the function $F$ can be treated as uniformly Lipschitz in the $\mu$-variable.

We also observe that the uniform bound of $\mu$ implies a  uniform bound for $Dv$. Indeed, one can first proceed as in  \rife{tilv}  to deduce that $\|\tilde v(t)\|_2$ is uniformly bounded, then by Lemma \ref{lem72} and the regularizing effect of the equation it follows that 
\be\label{dvk} 
\|Dv(t)\|_\infty \leq K\qquad  \forall t>0
 \ee
 for some constant $K>0$.
We now show the following  properties of $\Phi_\vep$:

(i) $\Phi_\vep$ is continuous and compact. 

To show this fact, let $\mu_n$ be bounded in $X$. By decay properties of the equation of $v$, and by Lipschitz continuity of $F$,
we have
$$
e^{-\de t} \| D v_n(t)\|_2 \leq c \int_t^{\infty} e^{-\nu(s-t)} e^{-\de s} \| \mu_n(s) e^{-\vep s}\|_2 ds
$$
which implies
$$
\| D v_n (t)\|_2 \leq c\, e^{-\vep t}   \,.
$$
Since
$$
\| \rho(t)\|_2^2 \leq c e^{-2\nu(t-t_0)} \| \rho(t_0)\|_2^2+  c \int_{t_0}^t \|Dv(s)\|^2ds 
$$
by choosing $t_0= \frac t2$ we get
$$
\| \rho_n(t)\|_2  \leq c e^{-\nu \frac t2}  +  c\,   e^{-\vep \frac t2} \,.
$$
In particular, $\rho_n(t)$ is uniformly small in $\elle 2$ for $t$ large. Since $\rho_n(t)$ is (locally in time) relatively compact for the uniform topology,  we deduce that it is compact in $L^\infty((0,\infty); L^2(\T))$.
The continuity is easily proved in a  similar way.

(ii) There exists $M>0$ such that, for any $\sigma\in [0,1]$, any solution of  $\mu= \sigma \Phi_\vep \mu$ satisfies the estimate $\|\mu\|\leq M$. This is consequence of  Lemma \ref{apriori-dis}; indeed, if $\mu= \sigma \Phi_\vep (\mu)$, then $(\mu, v)$ is a solution of \rife{lin-dis} with $h(x,p):= H(x,D  \bar u_\de(x) + p) - H(x,D \bar u_\de(x))
$, $f(x,\mu):= F(x,  \bar m_\de(x) + \mu) - F(x,  \bar m_\de(x))$ and $B(x,p):=   \bar m_\de(x) \left[ H_p(x,D \bar u_\de(x) + p)- H_p(x, D  \bar u_\de(x))\right]$. Using the global bound for $\mu$ and \rife{dvk}, and due to \rife{fK}, \rife{Hk}, the conditions   \rife{h}-\rife{B} are satisfied. Moreover, we have $\mu\geq -\sigma \bar m_\de$, which implies that \rife{summa2} holds true. Applying Lemma \ref{apriori-dis} we deduce that 
\be\label{dallemma}
\|\mu(t)\|_2\leq M\, e^{-\omega t} \|\mu_0\|_2 \quad \forall t>0\,.
\ee
In particular, $\|\mu(t)\|_2$ is uniformly bounded. 

After (i) and (ii), we can apply Schaefer's fixed point theorem (\cite[Thm 11.3]{GT}) to conclude with the existence of a fixed point $\mu^\vep$, depending on $\vep$. However, the estimate \rife{dallemma} is uniform in $\vep$, so we have that $\mu^\vep$ is uniformly bounded in $L^\infty((0,\infty); L^2(\T))$ and is actually uniformly decaying as $t\to \infty$. A similar estimate is deduced for $\tilde v^\vep(t), Dv^\vep(t)$. Then, by   compactness (as in point (i) above) we conclude that $(\mu^\vep, v^\vep)$ converges towards a solution $(\mu,v)$
corresponding to  a fixed point for $\Phi_0$. Hence $u= v+ \bar u_\de, m= \mu+ \bar m_\de$ yield a solution of \rife{ih}. In addition, we also deduce the estimate
\be\label{expede2}
\| m(t)-  \bar   m_\de\|_2+ \| Du(t)-D  \bar u_\de\|_2 \leq K\, e^{-\omega t} \|m_0-  \bar m_\de\|_2 \qquad \forall t>0 \,. 
\ee
Since $m_0\in \elle\infty$, and using  the local regularizing effect of the two parabolic equations, this estimate can be upgraded into  \rife{expde}.

Finally, we show that there is a unique solution to \rife{ih}. Indeed, if $(\tilde u, \tilde m)$ is another solution, then $(\tilde u-\bar u_\de,  \tilde m - \bar m_\de)$ is  a solution to  \rife{lin-dis} (with $\sigma=1$) where the functions $f, h, B$ are defined as above in step (ii). From Lemma \ref{apriori-dis} we deduce that $ (\tilde u, \tilde m)$ also satisfies \rife{expede2}. In particular this implies that $Du-D\tilde u \in L^2((0,\infty);\elle2)$. Therefore, we can repeat for $u-\tilde u$ and $m-\tilde m$ the same arguments which were used   in   Lemma \ref{apriori-dis}  to obtain \rife{sdv}. But in this case we have $(m-\tilde m)(0)=0$, so we get
$$
\int_0^{\infty}  \into e^{-\de s}    |Du-D\tilde u|^2\leq 0
$$
which implies $Du= D\tilde u$. Then $m=\tilde m$ from the second equation and, in turn, we get $(u-\tilde u)_t= \de (u-\tilde u)$. Since the two functions are bounded, this implies $u=\tilde u$.
\qed

\subsection{Vanishing discount limit}

Now we wish to close the chain of implications by showing what happens in the vanishing discount limit. 
The preliminary result which is needed is the asymptotic behavior of problem \rife{ergde}.

\begin{prop}\label{ergdisc} Assume that $H$ satisfies  \rife{Hk} and \rife{lip}, and that  $F(x,m)$  satisfies 
\rife{fK}, \rife{f3} and is differentiable with respect to $m$. There exists $\gamma_0$, only depending on $H,F$ such that, if $\gamma<\gamma_0$ in \rife{f3}, then the sequence $(\bar u_\de, \bar m_\de)$ solution of \rife{ergde} has the following asymptotic behavior as $\de\to 0$:
$$ 
\bar u_\de -  \frac {\bar \lambda} \de \mathop{\to}^{\de \to 0} \bar u + \theta\,,\qquad \quad \bar m_\de \mathop{\to}^{\de \to 0} \bar m \qquad \hbox{locally uniformly in $\T$}
$$
where $(\bar \lambda, \bar u, \bar m)$ is the unique solution of \rife{MFGergo} and $\theta$ is the unique constant such that the following ergodic stationary problem admits a solution $(w,\rho)$:
\be\label{thetapb}
\begin{cases}
\theta+ \bar u  - \kappa\Delta w + H_p(x, D\bar u)Dw= F_m(x,\bar m) \rho\,, & \hbox{in $\Omega$}
\\
 - \kappa \Delta \rho -\dive( \rho\, H_p(x,D\bar u))- \dive(\bar m \, H_{pp}(x,D\bar u) Dw)= 0 & \hbox{in $\Omega$.}
\end{cases}
\ee
\end{prop} 

We admit for a while the above result and we proceed with the proof of the vanishing discount limit of the infinite horizon problem.

\begin{thm}\label{vanlim} Under the assumptions of Proposition \ref{ergdisc}, let $(u_\de, m_\de)$ be the solution of \rife{ih}.  As $\de\to 0$, we have 
$$
u_\de- \frac{\bar \lambda}\de \to   v \,\,\,; \qquad \quad m_\de\to \mu 
$$
where $(v,\mu)$ is the unique solution of 
$$
\begin{cases}
-v_t + \bar \lambda - \kappa\Delta v + H(x, D v)= F(x,\mu )\,,\,\,\qquad\hbox{$t\in (0,\infty)$} &  
\\
\mu_t- \kappa \Delta \mu -\dive( \mu\, H_p(x,D v))= 0\,,\,\, \qquad \hbox{$t\in (0,\infty)$}& 
\\
\mu(0)= m_0 \,, & \\
v\in L^\infty((0,\infty)\times \T)\,,\,
\,\, Dv\in D\bar u + L^2((0,\infty);\elle2)\,,\, \,\,  \,
\lim\limits_{t\to \infty} \int v(t)dx = \theta &  
\end{cases}
$$
where   the constant $\theta$ is  the unique ergodic  constant of problem \rife{thetapb}. Moreover, we have
\be\label{commute}
v(t,x) \mathop{\to}^{t\to \infty} \bar u(x) + \theta\,\,\,, \,\,\,  \mu(t,x) \mathop{\to}^{t\to \infty} \bar m(x)\qquad \hbox{uniformly in $\T$.}
\ee
\end{thm}

\proof  Using \rife{expde}, and assumption \rife{fK}, we know that there exists  a constant $C$  such that
$$
\|F(x,m_\de)-F(x,\bar m_\de)\|_\infty \leq C \, e^{-\omega t}\qquad \forall t>0\,.
$$
We deduce that, for a convenient constant $M$, the functions $\bar u_\de  \pm M e^{-\omega t}$ are, respectively, super and subsolution of the equation 
$$
-u_t + \de u -   \kappa \Delta u + H(x,Du)= F(x,m_\de)\qquad t\in (0,\infty)\,.
$$
Using the comparison principle  (in the class of bounded solutions) for the viscous Hamilton-Jacobi equation, we obtain that
\be\label{stimaunif}
\bar u_\de - M e^{-\omega t}\leq u_\de \leq  \bar u_\de + M e^{-\omega t}\qquad \forall t>0\,.
\ee
Now we define
$$
v_\de(t,x):= u_\de(t,x)  -\frac{\bar \lambda} \de\,.
$$
Due to \rife{stimaunif} and Proposition \ref{ergdisc}, we deduce that $v_\de$ is uniformly bounded in $(0,\infty)\times \T$. From estimate \rife{expde}, we also know that $Dv_\de$ is uniformly bounded, and so is $m_\de$ as well. By local compactness and stability of the MFG system, there exists
a subsequence - not relabeled - and a couple $(v, \mu)$ such that  $v_\de$ converges to $v$, $m_\de$ converges to $\mu$  (locally uniformly) and $(v,\mu)$ is a solution of 
\rife{MFG-infty}. Notice that $Dv -D\bar u \in L^2((0,\infty)\times \T)$ as a consequence of  \rife{expede2}.  Finally, \rife{stimaunif} and Proposition \ref{ergdisc} imply that
$$
|v(t,x) - \bar u (x) - \theta | \leq C \, e^{-\omega t}  \mathop{\to}^{t\to \infty} 0
$$
and in particular 
\be\label{averthe}
\lim\limits_{t\to \infty} \into v(t)  = \theta\,.
\ee
Therefore $(v,\mu)$ is the unique solution of \rife{MFG-infty} satisfying \rife{averthe}.  We deduce that the whole sequence $(v_\de, m_\de)$ converges and this concludes the proof.
\qed

We point out that Theorem \ref{discex}, Proposition \ref{ergdisc} and Theorem \ref{vanlim} establish that the two limits, for time going to infinity and discount factor going to zero, actually commute.

We are only left with the proof of Proposition  \ref{ergdisc}, which is similar to  \cite[Prop 6.5]{CaPo}.

\vskip1em

{\bf Proof of Proposition \ref{ergdisc}.}
\quad 
Let $(\bar u_\de, \bar m_\de)$ be solutions of  \rife{ergde}.  We set
 $$
 w_\de:= \frac{\bar u_\de- \frac{\lambda} \de - \bar u}\de\,, \qquad \mu_\de= \frac{\bar m_\de-\bar m}\de
 $$
 and we verify that $(w_\de, \mu_\de)$ solves the problem
$$
\begin{cases}
  \de w_\de+ \bar u  -   \kappa \Delta w_\de + \frac{H(x,D\bar u + \de Dw_\de)- H(x,D\bar u)}\de= \frac{F(x,\bar m+ \de \mu_\de)- F(x, \bar m)}\de & 
\\
 -  \kappa  \Delta \mu_\de- \dive(\mu_\de\, H_p(x,D\bar u + \de Dw_\de)) = \dive\left( \bar m  \frac{[H_p(x,D\bar u + \de Dw_\de)- H(x,D\bar u)]}\de \right) & 
\\
\into \mu_\de=0    \,. 
& \end{cases}
$$
 We rephrase this problem as the   system 
\be\label{stat}
 \begin{cases}
 \de w_\de+ \bar u - \kappa \Delta w_\de + h^\de(x, Dw_\de)  = f^\de(x,  \mu_\de)   & 
\\
  - \kappa \Delta \mu_\de- \dive(\mu_\de\, h_p^\de(x, Dw_\de)) =  \dive(B^\de(x,Dw_\de)) & 
\\
\into \mu= 0 \,.   
& \end{cases}
\ee
where we have   
\begin{align*}
& h^\de(x,p): = \frac{H(x,D\bar u + \de p)- H(x,D\bar u)}\de  \to H_p(x, D\bar u)\cdot p \,,
\\
& f^\de(x,\mu): =\frac{F(x,\bar m+ \de \mu )- F(x, \bar m)}\de \to  F_m(x, \bar m ) \mu
\\
&  B^\de(x,p): = \bar m  \frac{[H_p(x,D\bar u + \de p)- H(x,D\bar u)]}\de  \to  \bar m\, H_{pp}(x, D\bar u) \,p
\end{align*}
It is easy to see (due to \rife{lip}) that   $\|\bar m_\de\|_\infty$, and then $\|D\bar u_\de\|_\infty$,  are bounded independently of $\de$. Then we can use the local conditions
\rife{Hk}, \rife{fK} and we deduce that $f^\de$, $B^\de$ satisfy
 \be\label{f1}
  |f^\de(x,\mu_\de)| \leq C_1\, |\mu_\de|
 \ee
and
  \be\label{Be}
| B^\de (x,Dw_\de)| \leq C_2 \, |Dw_\de|\,,\qquad B^\de(x,Dw_\de) Dw_\de \geq  c_0 \, |Dw_\de|^2\,.
 \ee
Therefore, using the convexity of $h^\de$ and assumption \rife{f3},  we estimate
\begin{align*}
\de \into w_\de \mu_\de + \into \bar u\, \mu_\de & \geq \into f^\de(x,\mu_\de)\mu_\de + \into B^\de(x,Dw_\de) Dw_\de
\\
& \geq c_0 \into  |Dw_\de|^2 - \gamma \into \mu_\de^2 \,.
\end{align*}
From the second equation in \rife{stat} we infer (see e.g. \cite[Corollary 1.3]{CaPo}) that, for a constant  $C$ only depending on $\kappa, \|h_p\|_\infty$, 
\be\label{uff}
 \into  \mu_\de^2   \leq C \into |B(x,Dw_\de|^2  \leq C\, C_2^2 \into |Dw_\de|^2\,.
\ee
 Therefore, using also the Poincar\'e-Wirtinger inequality, we deduce
 $$
 c_0 \into  |Dw_\de|^2 \leq \gamma \, C\, C_2^2 \into |Dw_\de|^2 + \de \| Dw_\de\|_2 \, \|\mu_\de\|_2 + \|\bar u\|_2 \, \|\mu_\de\|_2\,.
 $$
 Using again \rife{uff} in the last two terms, we see that  there exists $\gamma_0>0$ such that if $\gamma<\gamma_0$,  and $\de$ is sufficiently small, we have  that $Dw_\de$, and in turn $\mu_\de$, are bounded in $\elle2$.
 Now, using that $f^\de, h^\de,  B^\de$ grow at most linearly with respect to  $\mu_\de$ and $w_\de$, respectively, we can use a bootstrap  regularity argument and we conclude that  $\mu_\de$, $Dw_\de$ are bounded in $\elle\infty$. 
  
From the bound of $B^\de(x,Dw_\de)$ and $h_p$, we can now deduce that $\mu_\de$ is bounded in $C^{0,\alpha}(\T)$ for some $\alpha>0$, hence it is relatively compact in the uniform topology. Similarly, there exists $w$ such that, up to subsequences, $w_\de- \lg w_\de \rg $ converges to $w$ uniformly and in $W^{1,p}(\T)$, for any $p<\infty$. 
Since $\de\|w_\de\|_\infty$ is bounded, overall we conclude that, for some constant $\theta\in \R$ and some subsequence (not relabeled), we have
\be\label{limde}
\hbox{$\de w_\de \to \theta$, $w_\de - \lg w_\de \rg \to w$, $\mu_\de \to \mu$, uniformly in $\T$,}
\ee
where $(\theta, w, \mu)$ is a  solution of  \rife{thetapb}.
 We only need to prove the uniqueness for this limit problem. This is the usual argument we just used  in \rife{stat}. In problem \rife{thetapb} we observe that $F_m(x,\bar m)\geq -\gamma$ due to \rife{f3} and that $H_{pp}(x,D\bar u) $ is bounded from below and from above. Therefore,  if 
 $(\hat \theta, \hat w, \hat \mu)$ is any other solution,   we have
$$
  c_0 \into \bar m\, |Dw-D\hat w|^2 \leq  \gamma \into  (\mu-\hat \mu)^2 \leq \gamma \, C\, C_0^2\, \|\bar m\|_\infty \into \bar m\, |Dw-D\hat w|^2\\
$$
where $c_0, C_0$ denote the bounds of $H_{pp}(x,D\bar u)$ from below and from above, and we estimated $\mu-\hat \mu$ as we did before.
Thus, if $\gamma$ is sufficiently small, we deduce that $Dw= D\hat w$, and the uniqueness follows (first of $w$, then of $\mu$ by the second equation, and finally of $\theta$ from the first equation). The uniqueness of the limit also implies the convergence of the whole sequence $w_\de, \mu_\de$. Finally,  we find that $\de w_\de \to \theta$, i.e. $\bar u_\de- \frac{\lambda} \de - \bar u \to \theta$. So the Proposition is proved.
\qed


 \appendix
\section{Appendix}

We collect here  global in time  decay estimates of both viscous Hamilton-Jacobi and Fokker-Planck equations, which we used throughout the paper. 

We start with estimates on the viscous Hamilton-Jacobi equation, which can be found in the Appendix of \cite{CLLP2}.

\begin{lem}\label{lem72}  For given $V \in L^\infty ((0,T) \times \Omega)$ and $v_0 \in L^2  (\T)$, let  $v$ be
the solution of
\be\label{73}
\begin{cases}
 -v_t - \kappa \Delta v + Dv \cdot V = f & \hbox{in $(0, T ) \times \Omega$,}
 \\
 v(T)=v_0 \,.& 
 \end{cases}
 \ee
Then there exist constants $\nu>0$ and $C>0$ (only depending on $\kappa, d, \|V\|_\infty$) such that

(i)  $\tilde v := v - \lg v\rg $ satisfies
 $$
 \|\tilde v(t)\|_2 \leq C\, e^{-\nu (T-t)} \,    \|\tilde v_0\|_2 + C \int_t^T\|f (s)\|_2 \,e^{-\nu(s-t)}\, ds\qquad \forall t \leq T\, . 
 $$

(ii) For every $0 < t < t_0<T$, we have
$$
 (t_0 - t)\|Dv(t)\|_2^2 \leq C\,  [(t_0 - t) + 1]\left\{\|\tilde v(t_0 )\|_2^2 +\|f\|^2_{L^2  ((t,t_0 )\times\Omega) }+\|\tilde v\|_{L^2  ((t,t_0 )\times\Omega)}^2\right\}\,,
$$
 and
 $$
 \int_t^{t_0} \into |Dv|^2 \leq C\, \|\tilde v(t_0 )\|_2^2 + C\int_t^{t_0}\into  [|f |^2 + |\tilde v|^2 ]\,. 
 $$
\end{lem}
\qed
\vskip0.4em

 Now we turn to the Fokker-Planck equation.
 
\begin{lem}\label{stimaRhoF}
Assume that
$V \in L^\infty((t_0, T) \times \T)$, $F \in L^2((t_0,T); L^2(\T))$, and let $\rho \in L^2((t_0,T) ; L^2_0(\T))$ be a solution to
\begin{equation}\label{rhoeq}
\rho_t - \kappa \Delta \rho - {\rm div}(\rho V) = {\rm div}(F) \qquad \text{in $(t_0, T) \times \T$}\,.
\end{equation}
Then there exists $\nu>0$ and a constant $C$ (only depending on $\kappa,d$ and $\|V\|_\infty$) such that
we have, for every $t > t_0\geq  0$:
\be\label{76}
\|\rho(t)\|_2^2 \leq C \, \left\{e^{-\nu(t-t_0)}\, \|\rho(t_0 )\|_2^2 + \int_{t_0}^t \into  |F |^2  \right\} 
\ee
Moreover, if $t_1$ is such that  $t_0 \le t_1 < T-1$,    
 then we have, for every $\de\geq 0$,
\be\label{77}
\int_{t_1}^T e^{-\de t}\|\rho(t)\|^2_2 dt \le C \left\{e^{-\de t_1} \|\rho(t_0)\|^2_2 e^{-2\nu (t_1-t_0)} + \int_{t_0}^T e^{-\de t}\|F(t)\|^2_2 dt \right\}
\ee
for some $C > 0$ depending on $\kappa, d,\|V\|_\infty$ and independent of $\de$.
\end{lem}

\begin{proof} Estimate \rife{76} is already proved in the Appendix  of \cite{CLLP2}. So we only need to prove \rife{77}.

For a fixed $t \in (t_0, T)$, we start with considering the solution $v$ of
\begin{equation}\label{veq}
\begin{cases}
- v_t - \kappa \Delta v + Dv \cdot V = 0 & \text{in $(t_0, t) \times \T$,} \\
v(t) = \frac{\rho(t)}{\| \rho(t) \|_2}.
\end{cases}
\end{equation}
Lemma \ref{lem72} yields
\begin{equation}\label{tildeve}
\|\tilde v(s)\|_2 \le C e^{-\nu(t-s)} \qquad \forall s \le t
\end{equation}
where $\nu > 0$ and $C > 0$ depends on $\kappa, \|V\|_\infty$ only. Throughout the proof, the value of $C$ may increase, but will be always independent of $t$. 
Still using Lemma \ref{lem72} we have
\[
\|Dv(s)\|_2^2 \le C\Big( \|\tilde v(s+1) \|_2^2 + \|\tilde v\|^2_{L^2((s, s+1)\times \T)}\Big) \qquad \forall s  \le t - 1.
\]
Hence, from the estimate on $\tilde v$ we deduce that
\begin{equation}\label{dtildeve}
\|Dv(s)\|_2 \le C e^{-\nu(t-s)} \qquad  \forall s \le t - 1.
\end{equation}

By duality between \eqref{rhoeq} and \eqref{veq} we have
\[
\|\rho(t)\|_2 = \int \rho(t_0) v(t_0) + \int_{t_0}^t \int {\rm div}(F) v(s) ds.
\]
Integrating by parts and using the fact that $\int \rho(t_0) = 0$ we get
\begin{equation}\label{eq122}
\|\rho(t)\|_2 \le  \|\rho(t_0)\|_2  \|\tilde v(t_0)\|_2 + \int_{t_0}^{t}  \|D v(s)\|_2 \|F(s)\|_2.
\end{equation}
For $t > t_0+1$ we can split the last integral on $(t_0, t-1)$ and $(t-1, t)$, and apply \eqref{tildeve}, \eqref{dtildeve}, so that
\[
\|\rho(t)\|_2 \le C \|\rho(t_0)\|_2  e^{-\nu (t-t_0)} + C \int_{t_0}^{t-1}  e^{-\nu(t-s)} \|F(s)\|_2  ds + \int^t_{t-1}  \|D v(s)\|_2 \|F(s)\|_2.
\]
Using H\"older's inequality we obtain
\begin{multline}\label{eq123}
\|\rho(t)\|^2_2 \le C \|\rho(t_0)\|^2_2  e^{-2\nu (t-t_0)} + C\left(\int_{t_0}^{t-1}  e^{-\nu(t-s)} ds \right)\left( \int_{t_0}^{t-1}  e^{-\nu(t-s)} \|F(s)\|^2_2  ds \right) \\ + \left(\int^t_{t-1}  \|D v(s)\|^2_2 ds \right) \left(\int^t_{t-1} \|F(s)\|^2_2 ds \right).
\end{multline}
On one hand $\int_{t_0}^{t-1}  e^{-\nu(t-s)} \le \nu^{-1} e^{-\nu}$, and on the other hand $\int^t_{t-1}  \|D v(s)\|^2_2 \le C$ again by Lemma \ref{lem72}.  In addition, we have $1\leq e^{\de(t-s)}$ for  any $s\leq t$ and $\de\geq 0$.
Therefore we get
\begin{multline}\label{eq123bis}
\|\rho(t)\|^2_2 \le C\left\{  \|\rho(t_0)\|^2_2  e^{-2\nu (t-t_0)} +     \int_{t_0}^{t-1}  e^{-\nu(t-s)} \|F(s)\|^2_2  ds    +   \int^t_{t-1} e^{\de (t-s)}\|F(s)\|^2_2 ds  \right\}.
\end{multline}
Now we multiply  \eqref{eq123bis} by $e^{-\de t}$ and we integrate  in $(t_1 \vee t_0 + 1, T)$, obtaining  
\begin{multline*}
\int_{t_1 \vee t_0 + 1}^T e^{-\de t}\|\rho(t)\|^2_2 dt \le C \left\{ \|\rho(t_0)\|^2_2 \int_{t_1 \vee t_0 + 1}^T e^{-2\nu (t-t_0)-\de t}dt   \right.  \\ \left. +  \int_{t_1 \vee t_0 + 1}^T e^{-\de t}\int_{t_0}^{t-1}  e^{-\nu(t-s)} \|F(s)\|^2_2  ds dt + \int_{t_1 \vee t_0 + 1}^T  \int^t_{t-1} e^{-\de s}\|F(s)\|^2_2 ds dt \right\},
\end{multline*}
and by exchanging the order of integration we easily get
\[
\int_{t_1 \vee t_0 + 1}^T e^{-\de t}\|\rho(t)\|^2_2 dt \le C  \left\{ e^{-\de t_1}\|\rho(t_0)\|^2_2 e^{-2\nu (t_1-t_0)}  +  \int_{t_0}^{T-1} e^{-\de s}\|F(s)\|^2_2  ds + \int_{t_0}^T e^{-\de s}\|F(s)\|^2_2 ds \right\}
\]
for a possibly different constant $C$, depending on $\nu$ but  independent of $\de$. Going back to \eqref{eq122}, we get the remaining bound (if needed) on $\int_{t_1}^{t_0+1} \|\rho(t)\|^2_2$ by a straightforward application of \eqref{tildeve} and \eqref{dtildeve}.
\end{proof}

Finally, we include in this Appendix  a regularity lemma on the Fokker-Planck equation.

\begin{lem}\label{stimaRho1}
Let $\rho$ be a non-negative classical solution to
\[
\rho_t - \kappa \Delta \rho - {\rm div}(\rho V) = 0 \qquad \text{in $(0, T) \times \T$}.
\]
Then, for all $\eps > 0$ and $p' < \frac d{d-2}$
\[
\| \rho \|_{L^1((0,T); L^{p'}(\T))} \le \eps \int_0^T \int_{\T} |V|^2 \rho + C\left(1+ \frac 1 \eps\right) \|\rho(0)\|_{L^1(\T)} .
\]
for some $C > 0$ depending on $\kappa, p', T, d$.
\end{lem}

\begin{proof} Let $G(x,t)$ be the  kernel of $\partial_t(\cdot)-\kappa \Delta (\cdot)$ (on $\T$), and denote by $\star$ and $\star\star$ the space and space-time convolution respectively.
Then, the Duhamel representation formula yields
\[
\rho = \rho(0) \star G(t) + (\rho V) \star\star \, D_x G.
\]
For $q>1$ and $\gamma \in (1,2)$ to be chosen below, Young's inequality for convolutions yields
\begin{multline}\label{eq1423}
\|\rho\|_{L^1((0,T); L^{p'}(\T))} \le \|\rho(0)\|_{L^1(\T)} \|G\|_{L^1((0,T); L^{p'}(\T))} \\ + C \|\rho V\|_{L^{1}\big((0,T); L^{\frac{p'\gamma}{p'+\gamma-1}}(\T)\big)}
\|D_x G\|_{L^{1}\big((0,T); L^{\frac{p'\gamma}{p' \gamma - p' +1}}(\T)\big)}.
\end{multline}
By H\"older's inequality
\begin{align*}
\|\rho V\|_{L^{1}\big((0,T); L^{\frac{p'\gamma}{p'+\gamma-1}}(\T)\big)} & \le \|\rho^{\frac1\gamma} V\|_{L^\gamma((0, T) \times \T)} \|\rho^{\frac1{\gamma'}}\|_{L^{\gamma'}\big((0,T); L^{p'\gamma'}(\T)\big)} \\
& = \left(\int_0^T \int_{\T} |V|^\gamma \rho \right)^{\frac1\gamma} \|\rho\|^{\frac1{\gamma'}}_{L^{1}\big((0,T); L^{p'}(\T)\big)}.
\end{align*}
Moreover, $G$ and $D_xG$ are bounded in $L^1((0,T); L^{p'}(\T))$ and $L^{1}\big((0,T); L^{\frac{p'\gamma}{p' \gamma - p' +1}}(\T)\big)$ respectively, provided that
\[
\frac{2}{dp'}-\frac d 2 + 1 > 0, \qquad \frac d 2 \frac{p' \gamma + p' -1}{p'\gamma} - \frac d 2 +\frac 1 2 >0,
\]
that is equivalent to
\[
p' < \frac d{d-2}, \qquad p'< \frac d{d-\gamma}.
\]
The first inequality is true by the standing assumptions. Pick then $\gamma < 2$ so that $p'< \frac d{d-\gamma} < \frac d{d-2}$. Then,
\begin{multline*}
\|\rho(0)\|_{L^1(\T)} \|G\|_{L^1((0,T); L^{p'}(\T))} + C \|\rho V\|_{L^{1}\big((0,T); L^{\frac{p'\gamma}{p'+\gamma-1}}(\T)\big)}
\|D_x G\|_{L^{1}\big((0,T); L^{\frac{p'\gamma}{p' \gamma + p' -1}}(\T)\big)} \\ \le
C_1 \|\rho(0)\|_{L^1(\T)} + C_1 \left(\int_0^T \int_{\T} |V|^\gamma \rho \right)^{\frac1\gamma} \|\rho\|^{\frac1{\gamma'}}_{L^{1}\big((0,T); L^{p'}(\T)\big)}.
\end{multline*}
Using Young's inequality, and plugging back into \eqref{eq1423}, we obtain
\[
\|\rho\|_{L^1((0,T); L^{p'}(\T))} \le C_2 \|\rho(0)\|_{L^1(\T)} + C_2 \left(\int_0^T \int_{\T} |V|^\gamma \rho \right).
\]
A further application of Young's inequality yields, for any $\eps > 0$,
\[
\|\rho\|_{L^1((0,T); L^{p'}(\T))} \le C_2 \|\rho(0)\|_{L^1(\T)} + \eps \left(\int_0^T \int_{\T} |V|^2 \rho \right)+ \frac {C_2^{\frac2{2-\gamma}}} \eps \left(\int_0^T \int_{\T} \rho \right) ,
\]
and since $\int_{\T} \rho (t) = \int_{\T} \rho (0)$ for all $t$, we conclude.
\end{proof}

{\bf Acknowledgement.} This research was partially supported by Indam Gnampa project 2019.  We also thank the anonymous reviewers for their careful reading of the first version of the paper, and for their suggestions of related references.


\begin{thebibliography}{99}

\bibitem{Achdou-CIME} Y. Achdou and M. Lauri\`ere, {\it Mean Field Games and Applications: Numerical Aspects}, in: {Mean field games}, Lectures Notes in Mathematics (CIME - series), Springer, (2021), to appear.

\bibitem{Ambrose} D.  M. Ambrose {\it Existence theory for non-separable mean field games in Sobolev spaces}, arXiv:1807.02223v2, 2020.


\bibitem{CLLP1}
{P.~Cardaliaguet, J.-M. Lasry, P.-L. Lions, and A.~Porretta}, {\em Long
  time average of mean field games.}, Networks \& Heterogeneous Media, 7
  (2012), 279--301.

\bibitem{CLLP2}
{P.~Cardaliaguet, J.-M. Lasry, P.-L. Lions, and A.~Porretta}, {\em Long
  time average of mean field games with a nonlocal coupling}, SIAM Journal on
  Control and Optimization, 51 (2013), 3558--3591.


 
\bibitem{CaPo}
{P.~Cardaliaguet and A.~Porretta}, {\em Long time behavior of the master
  equation in mean field game theory}, Analysis \& PDE, 12 (2019),
1397--1453.
  
  \bibitem{CP-CIME} P. Cardaliaguet and A. Porretta, {\it An introduction to Mean Field Game theory}, in: {Mean field games}, Lectures Notes in Mathematics (CIME - series), Springer, (2021), to appear.
  
\bibitem{CeCi} A. Cesaroni, M. Cirant, {\it Brake orbits and heteroclinic connections for first order Mean Field Games}. Trans. Amer. Math. Soc., 374-7(2021), 5037--5070.

\bibitem{CCPDE} M. Cirant, {\it Stationary focusing mean-field games}. Comm. Partial Differential Equations, 41(8) (2016), 1324--1346.

\bibitem{cjde} M. Cirant, {\it On the existence of oscillating solutions in non-monotone Mean-Field Games}. J. Differential Equations, 266-12(2019), 8067--8093.

\bibitem{CG} M. Cirant, A. Goffi {\it On the problem of maximal $L^q$-regularity for viscous Hamilton-Jacobi equations}, arXiv:2001.11970, 2020.

\bibitem{CiGolast} M. Cirant, A. Goffi {\it Maximal $L^q$. regularity for parabolic Hamilton-Jacobi equations and applications to Mean Field Games},  arXiv:2007.14873, 2020.

\bibitem{CGhi} M. Cirant, D. Ghilli, {\it Existence and non-existence for time-dependent mean field games with strong aggregation},  arXiv:2011.00798, 2020.

\bibitem{Ci-To} {M. Cirant  and D. Tonon},  {\it Time-Dependent Focusing Mean-Field Games: The Sub-critical Case} J. Dyn.  Diff. Eq. 31 (2019), 49--79. 

\bibitem{Grune+al} Damm, T., Gr\"une, L., Stielerz, M., Worthmann, K. {\it An exponential turnpike theorem for dissipative discrete time optimal control problems} SIAM J. Control. Optim. 52 (2014), 1935--1957.

\bibitem{DSS} Dorfman, R., Samuelson, P.A., Solow, R.: {\sl Linear Programming and Economic Analysis}, McGraw-Hill, New York, 1958.



\bibitem{GT}
{D.~Gilbarg and N.~S. Trudinger}, {\em Elliptic partial differential
  equations of second order}, Springer, 2015.


\bibitem{gomes2010discrete}
{D.~A. Gomes, J.~Mohr, and R.~R. Souza}, {\em Discrete time, finite state
  space mean field games}, Journal de math{\'e}matiques pures et
  appliqu{\'e}es, 93 (2010), 308--328.
  
  \bibitem{GoSe} D. Gomes, M. Sedjro, {\it One-dimensional, forward--forward mean-field games with congestion}. Discrete Contin. Dyn. Syst. Ser. S 11 (5) (2018) 901--914.
  
 \bibitem{Gomes-book} D.~A. Gomes, E.~A. Pimentel, and V.~Voskanyan.  {\sl Regularity theory for
  mean-field game systems}, Springer Berlin, 2016.
  
  \bibitem{HP} M. Hieber, J. Pr\"uss. {\em Heat kernels and maximal Lp-Lq estimates for parabolic evolution equations}, Comm. Partial Differential Equations, 22(9-10) (1997), 1647--1669.
  
  
  \bibitem{HCM} M. Huang, P. Caines and R. Malham\'e, {\em Large population
  stochastic dynamic games: closed-loop mckean-vlasov systems and the Nash
  certainty equivalence principle}, Communications in Information \& Systems, 6
  (2006),  221--252.

  
  \bibitem{LSU} Ladyzenskaja O.A., Solonnikov V.A and Ural'ceva N.N: {\sl Linear and quasi-linear equations of parabolic type}. Translations of Mathematical Monographs, Vol. 23 American Mathematical Society, Providence, R.I. 1967.
  
  \bibitem{LadyElliptic} Ladyzhenskaya, O.A. , N.N. Ural'ceva, {\sl Linear and Quasilinear Elliptic Equations}, Mathematics in Science and Engineering, Vol.46,  Academic Press, New York, London 1968.

\bibitem{LL06cr1}
{J.-M. Lasry and P.-L. Lions}, {\em Jeux {\`a} champ moyen. I --le cas
  stationnaire}, Comptes Rendus Math{\'e}matique, 343 (2006),  619--625.

\bibitem{LL06cr2}
{J.-M. Lasry and P.-L. Lions}, {\em Jeux {\`a} champ
  moyen. II --horizon fini et contr{\^o}le optimal}, Comptes Rendus
  Math{\'e}matique, 343 (2006),  679--684.

\bibitem{LL07}
{J.-M. Lasry and P.-L. Lions}, {\em Mean field games},
  Japanese journal of mathematics, 2 (2007),  229--260.
  
  \bibitem{L-college}
{P.-L. Lions}, {\em Cours au college de france}, 2007--2012.

  
  
  \bibitem{masoero2019long}
{M.~Masoero}, {\em On the long time convergence of potential MFG},
  Nonlinear Differential Equations and Applications NoDEA, 26 (2019), 15.
  
  \bibitem{P-MTA} {A.~Porretta}, {\it On the turnpike property in mean field games}, {Minimax Theory and Appl.} 3 (2018), 285--312.

\bibitem{PZ}  {A. Porretta and E. Zuazua},   {\it Long time versus steady state optimal control},  Siam J. Control Optimization  51  (2013), 4242--4273.

\bibitem{Tran} H. V. Tran, {\it A Note on Nonconvex Mean Field Games}, Minimax Theory and its Applications 3-2 (2018), 323--336.

 \bibitem{TZ} Tr\'elat, E., Zuazua, E., {\it The turnpike property  in finite-dimensional nonlinear optimal control}, J. Diff. Equations  258 (2015), 81--114.


\bibitem{Zu} E. Zuazua, {\it Large time control and turnpike properties for wave equations},   Annual Reviews in Control, 44 (2017) 199--210.

\end{thebibliography}
\end{document}